\documentclass{amsart}
\usepackage{amssymb, latexsym, amsthm, enumitem, color, amsmath, tikz-cd, mathtools, bbm, mathrsfs}

\usepackage[unicode=true,pdfusetitle,bookmarks=true,bookmarksnumbered=false,
bookmarksopen=false,breaklinks=true,pdfborder={0 0 0},
pdfborderstyle={},backref=false,colorlinks=true]{hyperref}
\definecolor{myblue}{rgb}{0.09,0.32,0.44} %
\hypersetup{pdfborder={0 0 0},pdfborderstyle={},colorlinks=true,linkcolor=myblue,citecolor=myblue,urlcolor=blue}

\newtheorem{thm}{Theorem}[section]
\newtheorem*{thm*}{Theorem}

\newtheorem{assum}[thm]{Assumption}

\newtheorem{cor}[thm]{Corollary}
\newtheorem{defn}[thm]{Definition}

\newtheorem{fact}[thm]{Fact}
\newtheorem{lem}[thm]{Lemma}
\newtheorem{prop}[thm]{Proposition}

\theoremstyle{remark}
\newtheorem{rem}[thm]{Remark}

\newtheorem*{rem*}{Remark}
\newtheorem*{rems*}{Remarks}

\newcommand\Cref[1]{{Corollary~\ref{#1}}}

\newcommand\Lref[1]{{Lemma~\ref{#1}}}
\newcommand\Pref[1]{{Proposition~\ref{#1}}}
\newcommand\Rref[1]{{Remark~\ref{#1}}}
\newcommand\Tref[1]{{Theorem~\ref{#1}}}
\newcommand\Sref[1]{{\S \ref{#1}}}

\newcommand{\Z}{\mathbb{Z}}
\newcommand{\Q}{\mathbb{Q}}
\newcommand{\R}{\mathbb{R}}
\newcommand{\C}{\mathbb{C}}
\newcommand{\FF}{\mathbb{F}}
\newcommand{\FFF}{\mathbb{F}^{\circ}}
\newcommand{\PP}{\mathcal{P}}
\renewcommand{\AA}{\mathcal{A}}
\newcommand{\BB}{\mathcal{B}}
\newcommand{\DD}{\mathcal{D}}
\newcommand{\set}[1]{\left\{#1\right\}}
\newcommand{\sub}{\subseteq}
\renewcommand{\Pr}{\mathbb{P}}

\newcommand{\E}{\mathbb{E}}
\newcommand{\eps}{\varepsilon}
\newcommand{\floor}[1]{\left\lfloor #1 \right\rfloor}
\newcommand{\ceil}[1]{\left\lceil #1 \right\rceil}
\newcommand{\suchthat}{\,\ifnum\currentgrouptype=16\middle\fi|\,}

\newcommand{\supp}{\operatorname{supp}}
\newcommand{\Gal}{\operatorname{Gal}}
\newcommand{\Alt}{\operatorname{Alt}}
\newcommand{\Sym}{\operatorname{Sym}}
\newcommand{\disc}{\operatorname{disc}}
\newcommand{\tr}{\operatorname{tr}}
\newcommand{\norm}[1]{\| #1 \|}

\newcommand{\ind}[1]{\mathbbm{1}_{#1}}
\newcommand{\Res}{\operatorname{Res}}
\renewcommand{\Re}{\operatorname{Re}}

\newcommand{\ord}{\operatorname{ord}}
\newcommand{\Var}{\operatorname{Var}}
\newcommand{\Cov}{\operatorname{Cov}}

\newcommand{\ad}{\mathrm{ad}}
\newcommand{\nad}{\mathrm{nad}}
\newcommand{\Ckon}{C_0}
\newcommand{\Cord}{C_1}
\newcommand{\ksm}{k_{\mathrm{sm}}}

\newcommand{\dfi}[2]{\left(\genfrac{}{}{0pt}{}{#1}{#2}\right|}
\newcommand{\df}[2]{\mathchoice{{\textstyle \dfi{#1}{#2}}}
  {\dfi{#1}{#2}}{\dfi{#1}{#2}}{\dfi{#1}{#2}}}
\newcommand{\whi}[2]{\widehat{#1}#2}
\newcommand{\whnu}[1]{\whi{\nu}{#1}}
\newcommand{\whrho}[1]{\whi{\rho}{#1}}
\newcommand{\spmod}[1]{\;(\mathrm{mod}\;#1)}

\makeatletter
\def\moverlay{\mathpalette\mov@rlay}
\def\mov@rlay#1#2{\leavevmode\vtop{%
   \baselineskip\z@skip \lineskiplimit-\maxdimen
   \ialign{\hfil$\m@th#1##$\hfil\cr#2\crcr}}}
\newcommand{\charfusion}[3][\mathord]{
    #1{\ifx#1\mathop\vphantom{#2}\fi
        \mathpalette\mov@rlay{#2\cr#3}
      }
    \ifx#1\mathop\expandafter\displaylimits\fi}
\makeatother

\newcommand{\bigcupdot}{\charfusion[\mathop]{\bigcup}{\mbox{\larger[2]$\cdot$}}}

\title{Galois groups of random polynomials of large degree}

\author{Guy Blachar}
\address{Universit\'{e} Paris-Dauphine -- CEREMADE, Place de Lattre de Tassigny, Paris, France}
\email{guy.blachar@gmail.com}
\author{Emmanuel Breuillard}
\address{Mathematical Institute, Oxford OX1 3LB, United Kingdom}
\email{emmanuel.breuillard@maths.ox.ac.uk}
\author{Gady Kozma}
\address{The Weizmann Institute of Science, Rehovot, Israel}
\email{gady.kozma@weizmann.ac.il}

\begin{document}

\begin{abstract}
    We study random polynomials of the form $R(x)=x^n+\omega_{n-1}x^{n-1}+\cdots+\omega_0$, where $\omega_0,\dots,\omega_{n-1}$ are independent, uniformly bounded integer-valued random variables, and $\omega_1,\dots,\omega_{n-1}$ have a fixed common law $\mu$. %

    We prove (unconditionally) that, if the R\'{e}nyi entropy of order $2$ satisfies $H_2(\mu)=-\log\norm{\mu}_2^2>12$, then $\Pr(\disc(R)\text{ is a square})=O_\mu(1/\log n)$. Combined with previous results, this shows that, for such measures $\mu$ and under additional hypotheses, $R$ has full Galois group $\Sym(n)$ with high probability, when conditioned on $\omega_0\ne 0$.
\end{abstract}

\maketitle

\section{Introduction}

A classical problem is to determine the Galois group of a random polynomial with integer coefficients. In many natural models, one expects the polynomial to be irreducible and its Galois group to be the full symmetric group with probability tending to one.

In this work, we study the \emph{small box model}. We fix a finitely supported non-degenerate probability measure $\mu$ on $\Z$, and let
\[
R(x) = x^n + \omega_{n-1}x^{n-1} + \cdots + \omega_0 \in \Z[x]
\]
be our random polynomial, where $\omega_0,\dots,\omega_{n-1}$ are independent, and $\omega_1,\dots,\omega_{n-1}$ have common law $\mu$. We allow $\omega_0$ to have a different law, in particular to allow conditioning on $\omega_0\ne 0$. We are interested in the arithmetic behaviour of the polynomial when its degree $n$ tends to infinity.

This model was previously studied in several works. It was established in a work of Breuillard and Varj\'{u} \cite{BreuillardVarju19} under the generalized Riemann hypothesis for Dedekind zeta functions of number fields, and in %
\cite{BarySorokerKozma20,BarySorokerKoukoulopoulosKozma23} unconditionally, that in many settings, such polynomials are irreducible with high probability, and their Galois group contains the alternating group $\Alt(n)$ with high probability.

The remaining question in all these works is whether one can distinguish $\Alt(n)$ from~$\Sym(n)$. To answer that, we recall that when $R$ is separable, its Galois group is contained in $\Alt(n)$ if and only if its discriminant is a square in $\Z$. We answer this unconditionally on the generalized Riemann hypothesis, when $\mu$ has large enough entropy:

\begin{thm}\label{thm:main}
Assume that $H_2(\mu)>12$, and let $H\ge 1$ such that $|\omega_j|\le H$ for all~$j$ almost surely. Then there are constants $C=C(\mu)$ and $n_0=n_0(\mu,H)$ such that
\[
  \Pr(\disc(R)\text{ is a square}) \le \frac{C(\mu)}{\log n}
\]
for all $n\ge n_0$.
\end{thm}

Here $H_2(\mu)=-\log\norm{\mu}_2^2$ stands for the R\'{e}nyi entropy of order $2$. We do not believe that the entropy assumption is necessary for the theorem to hold, but it is an outcome of the strategy we use.

In fact, our proof shows that with high probability, %
there is a prime~$p$ such that $p|\disc(R)$ but $p^2\nmid\disc(R)$.

\begin{cor}\label{cor1}
Let $N>\exp(12)$ be an integer and assume that $\mu$ is supported uniformly on $N$ consecutive integers in $\Z$. Then there is a constant $C'(\mu)>0$ such that
\[
  \Pr(R\text{ is irreducible and has Galois group }\Sym(n) \suchthat \omega_0\ne 0) \ge 1- \frac{C'(\mu)}{\log n}
\]
for all sufficiently large $n$.
\end{cor}

Under the assumptions of Corollary \ref{cor1}, irreducibility of $R$ with high probability ($\ge 1-1/n^{c(\mu)}$) was proven in \cite{BarySorokerKoukoulopoulosKozma23} along with the fact that the Galois group $\Gal(R)$ is either $\Alt(n)$ or $\Sym(n)$. The corollary follows immediately from the theorem upon noticing that $\disc(R)$ must be a square if $\Gal(R)\leq \Alt(n)$. %

We also note that the assumption that $\mu$ is uniformly supported on $N>\exp(12)$ consecutive integers can be replaced by $\mu$ arbitrary with $H_2(\mu)>12$ if we further assume the generalized Riemann hypothesis for Dedekind zeta functions of number fields. This follows from \cite[Theorem 2]{BreuillardVarju19}.

\subsection*{Related works}

The questions mentioned above were already studied by Hilbert~\cite{Hilbert1892} and van der Waerden~\cite{vanderWaerden36}. In particular, van der Waerden showed that, if one samples a polynomial $R$ uniformly at random from the set of all monic polynomials of a fixed degree $n$ with non-leading coefficients in $[-L,L]\cap\Z$, then the polynomial is irreducible and has full Galois group $\Sym(n)$ with probability tending to $1$ as $L$ tends to $\infty$, i.e.,
\[
\lim_{L\to\infty}\Pr(\Gal(R)\ne\Sym(n))=0.
\]
This is known as the \emph{large box model}.

Several later works refined the estimates for $\Pr(\Gal(R)\ne\Sym(n))$ and studied the probabilities of the other possible Galois groups. Van der Waerden conjectured that the main contribution comes from polynomials with a rational root, whose Galois group is typically $\Sym(n-1)$. Chela~\cite{Chela63} showed that, for every fixed $n\ge3$, the probability of reducibility is asymptotic to $c_n/L$ for an explicit constant $c_n>0$. Gallagher~\cite{Gallagher73} used the large sieve to prove the bound
\[
\Pr(\Gal(R)\ne\Sym(n))=O_n\left(\frac{\log L}{\sqrt L}\right),
\]
which was later improved by Dietmann~\cite{Dietmann13}. More recently, Bhargava~\cite{Bhargava25} proved that, for every fixed $n\ge3$,
\[
\Pr(\Gal(R)\ne\Sym(n))=O_n\left(\frac{1}{L}\right).
\]
This determines the order of this probability, although it does not give the asymptotic formula predicted by van der Waerden. These results concern a fixed degree, and the dependence of the estimates on $n$ is important when considering models in which the degree also grows.

Random polynomials with independent coefficients have also been widely studied in probability theory, particularly in connection with the distribution of their roots; see e.g.\ \cite{BlochPolya31,Kac43}. Their arithmetic properties in the small box model have proved more difficult to understand. Here the degree tends to infinity while the coefficient measure remains fixed, so the large box estimates do not directly apply.

Odlyzko and Poonen~\cite{OdlyzkoPoonen93} studied polynomials with coefficients in $\set{0,1}$ and with leading and constant coefficients equal to $1$. They conjectured that a uniformly sampled polynomial of this form should be irreducible with high probability as its degree $n$ tends to $\infty$. Konyagin~\cite{Konyagin99} showed that such polynomials do not have factors of degree less than $\frac{cn}{\log n}$ with high probability, for some universal constant $c>0$. He also obtained a lower bound of order $1/\log n$ for the probability of irreducibility. These results exclude small factors, but still allow a factorization into two or more factors of large degree.

Konyagin's techniques were further developed by Breuillard and Varj\'{u}~\cite{BreuillardVarju19}. Under the generalized Riemann hypothesis for Dedekind zeta functions of number fields, they proved that random polynomials in the small box model are irreducible and their Galois group contains the alternating group $\Alt(n)$ with high probability. In particular, their results apply to every fixed non-degenerate, finitely supported coefficient measure on $\Z$, after conditioning the constant coefficient to be nonzero.

In \cite{BarySorokerKozma20} the same conclusion was obtained unconditionally under arithmetic restrictions on the coefficient measure. Their result applies, for example, to the uniform measure on $\set{1,\dots,L}$ when $L$ is divisible by at least four distinct primes. A similar method of testing 4 primes was used in  \cite{BarySorokerKoukoulopoulosKozma23} later to get analytic rather than arithmetic conditions on the distribution. In particular, they proved irreducibility and containment of $\Alt(n)$ with high probability when the coefficient measure is uniform on any fixed interval of at least $35$ integers, again conditioning the constant coefficient to be nonzero. They also showed that, for every fixed non-degenerate, finitely supported measure $\mu$, there is a constant $\theta=\theta(\mu)>0$ such that the random polynomial has no factors of degree at most $\theta n$ with high probability.

Thus, in many cases, these works reduce the question of determining the Galois group to distinguishing $\Alt(n)$ from $\Sym(n)$. However, this distinction requires an additional argument, even when irreducibility and containment of $\Alt(n)$ are already known. For a separable polynomial, it is precisely the question of whether its discriminant is a square.

The probability that a random polynomial has a double root, equivalently that its discriminant vanishes, has also been studied. Feldheim and Sen \cite{FeldheimSen17} obtained estimates for this probability for bounded integer-valued coefficient distributions whose largest atom has mass at most $1/2$. More recently, Michelen and Yakir \cite{MichelenYakir26} proved a limit law for the minimal distance between roots, which in particular implies the absence of double roots with high probability under their assumptions. In a recent work \cite{MichelenYakir25}, they also proved a law of large numbers for the discriminant of random polynomials.

In a recent work~\cite{BarySorokerGoldgraber25}, Bary-Soroker and Goldgraber studied random polynomials from the \emph{growing box model}, in which the non-leading coefficients are independent and uniform on $\set{-H_n,\dots,H_n}$, where $H_n\to\infty$ as the degree $n$ tends to infinity. They proved that the random polynomial has full Galois group $\Sym(n)$ with high probability, regardless of how slowly $H_n$ tends to infinity. The main new ingredient is an upper bound for the probability that the discriminant is a square. Their result therefore allows the coefficient range to grow arbitrarily slowly, but still requires it to grow. Our theorem concerns a fixed coefficient measure satisfying the entropy assumption stated above.

Related questions have also been studied for random reciprocal polynomials. Hokken \cite{Hokken26} obtained asymptotic counts of reciprocal and skew-reciprocal polynomials with coefficients in $\set{-1,1}$ and square discriminant. Hokken and Koukoulopoulos \cite{HokkenKoukoulopoulos25} studied irreducibility and Galois groups for more general coefficient distributions. They also showed that the probability of a nonzero square discriminant tends to zero when the free coefficients are sampled independently from a fixed finitely supported nondegenerate measure. In the reciprocal model, the symmetry of the coefficients gives a simpler expression for the discriminant up to a square, which is used in these arguments.

\subsection*{Idea of the proof and structure of the paper}

Our strategy for proving \Tref{thm:main} is to study the number of times in which a given prime $p$ divides the discriminant $\disc(R)$. This is closely related to the factorization of $R$ modulo $p$: as a first approximation, one can think that any double factor in this factorization, i.e.\ any $Q\in\FF_p[x]$ such that $Q^2\mid R$ over $\FF_p$, contributes $p^{\deg Q}$ to this discriminant. A precise statement would require us to also consider triple roots modulo $p$ and double roots modulo $p^2$, but this turns out to play only a technical part in the proof. This suggests the following approach: find a prime $p$ and $a\in\FF_p$ such that $(x-a)^2\mid R$ over $\FF_p$, and such that no other $Q\in\FF_p[x]$ satisfies $Q^2\mid R$. This will show that $p\mid\disc(R)$ but $p^2\nmid\disc(R)$, so $\disc(R)$ cannot be a square.

It is natural to conjecture that the probability of finding an $a\in\FF_p$ such that $(x-a)^2\mid R$ is approximately $\frac{1}{p}$, and that these events are roughly independent for different primes. Indeed, for a given ``generic'' $a\in\FF_p$, one could expect $\Pr(R(a)=R'(a)=0)\approx\frac{1}{p^2}$, which we then sum over $a\in\FF_p$ to get $\frac{1}{p}$, while the probability for the existence of double roots in larger extensions should be negligible compared to $\frac{1}{p}$. Since $\sum_p\frac{1}{p}$ diverges, we should be able to find such a prime if we search far enough. Unfortunately, showing that no other $Q$ satisfies $Q^2\mid R$ over $\FF_p$ seems beyond current techniques.

Indeed, to check whether an irreducible $Q\in\FF_p[x]$ divides $R$, the approach of all of \cite{BarySorokerKozma20,BarySorokerKoukoulopoulosKozma23,BreuillardVarju19} is to take a root $\alpha$ of $Q$ (which lies in $\FF_{p^{\deg Q}}$), consider $R(\alpha)$ as a Markov chain, and show that it \emph{mixes} by time $n=\deg R$. We follow an analogous approach and show that the pair $(R(\alpha),R'(\alpha))$ is mixed by time $n$ under similar conditions to the mixing of $R(\alpha)$ alone. This allows us to exclude double roots in extensions of $\FF_p$ which are not too large. However, this approach has a hard limit known as the entropy bound: the entropy of $R(\alpha)$ (and of $(R(\alpha),R'(\alpha))$) is not larger than the entropy of $R$ itself, so mixing can only happen if the entropy of $R$ is at least $\log(p^{\deg Q})$. This restricts the approach only to a constant number of primes, so the sum $\sum_p\frac{1}{p}$ will not guarantee the existence of a good prime $p$ as above.

A more refined approach is as follows. Assume you could get to the entropy bound for every prime $p$. This would show that with probability $\approx 1/p$ the polynomial $R$ has a double root in $\FF_{p}$ but no double root in $\FF_{p^k}$ for any $k\le H(\mu)n/\log p$, from which we would be able to conclude that $p^{\lceil H(\mu)n/\log p\rceil}\approx \exp(H(\mu)n)$ divides $\disc(R)$. Since $\sum_{p\le\exp(n)}1/p\approx\log n$, this would allow us to find divisors of $\disc(R)$ whose product is approximately
$\exp(H(\mu)n\log n)$. But Hadamard's bound gives $\disc(R)\le \exp(2n\log n + O(n))$, which would be a contradiction if $H(\mu)$ is sufficiently large and if $\disc(R)\ne 0$. This is our basic scheme.

The hardest part in performing this scheme is getting to the entropy bound. The best existing techniques, by Breuillard and Varj\'{u}~\cite{BreuillardVarju19} (based on ideas of Konyagin~\cite{Konyagin92}), get to within $\log^3 n$ of the entropy bound. Namely, we show using these techniques that $(R(\alpha),R'(\alpha))$ is mixed if $p^k\le\exp(n/\log^3 n)$. This is not enough for our purposes. Under GRH,~\cite{BreuillardVarju22} do get to the entropy bound, but only for $R(\alpha)$ and for most primes $p$ and most $\alpha\in\FF_p$ --- not for the couple $(R(\alpha),R'(\alpha))$. Getting to the entropy bound remains an open problem.

However, what we need is much weaker than full mixing. We need to estimate $\Pr(R(a)=R'(a)=R(\alpha)=R'(\alpha)=0)$ (which, under full mixing, would be approximately  $\frac{1}{p^{2k+2}}$). It is enough for us to show that it is significantly less than $\frac{1}{p^{k+2}}$. This will allow to sum over $a$ and $\alpha$, and get that the probability that $R$ has a double root in $\FF_p$ and another double root in some extension $\FF_{p^k}$ is significantly less than $\frac{1}{p}$.

The core argument uses two results from additive number theory. Denote by $\nu_{a,\alpha}$ the measure on $\FF_p^2\times\FF_{p^k}^2$ given by the distribution of $(R(a),R'(a),R(\alpha),R'(\alpha))$, and let $\nu^{(d,e)}_{a,\alpha}$ by the measure one gets by considering only the subsum $\omega_dx^d+\dots+\omega_ex^e$ of our random polynomial (so $\nu^{(0,n)}_{a,\alpha}=\nu_{a,\alpha})$. We first reduce the problem of bounding $\nu_{a,\alpha}(0,0,0,0)$ to a problem of bounding the $L^2$ norm $\norm{\nu^{(1,n/2)}_{a,\alpha}}_2$ using Cauchy-Schwarz. Let $m=\lfloor\frac{1}{6}n\rfloor$. We then use the fact that $\nu^{(1,2m)}=\nu^{(1/m)}*\nu^{(m+1,2m)}$ to claim that either $\norm{\nu_{a,\alpha}^{(1/2m)}}_2$ is much smaller than $\norm{\nu_{a,\alpha}^{(1,m)}}_2$, or it has a additive structure. Namely, we first apply the Balog-Szemer\'{e}di-Gowers theorem to claim that $\nu_{a,\alpha}^{(1,m)}$ is quite close (in a specific sense) to the uniform measure on some set $H\sub\FF_p^2\times\FF_{p^k}^2$, with $H$ satisfying that the Minkowski sum $H+H+H$ not much larger than $H$. We then apply a quasi-polynomial version of Freiman's theorem due to Sanders \cite{Sanders13} to conclude that $H$ is similar to a subspace. But this is impossible. A simple, but crucial argument shows that $\nu_{a,\alpha}^{(1,m)}(U)\le(1-c)^{\operatorname{codim} U}$ for any subspace~$U$.

The conclusion of this part is that either $\norm{\nu_{a,\alpha}^{(1,m)}}_2$ is small, or $\norm{\nu_{a,\alpha}^{(1,2m)}}_2$ is much smaller than $\norm{\nu_{a,\alpha}^{(1,m)}}_2$. Using that $\nu_{a,\alpha}^{(1,3m)}=\nu_{a,\alpha}^{(1,2m)}*\nu_{a,\alpha}^{(2m+1,3m)}$ and that $\norm{\nu_{a,\alpha}^{(2m+1,3m)}}_2=\norm{\nu_{a,\alpha}^{(1,m)}}_2$ allows to get a usable estimate for $\sum_{a,\alpha}\norm{\nu_{a,\alpha}^{(1,3m)}}_2$. This is the most important part of the argument and explains the number $12$ which appears in the statement of \Tref{thm:main}. Indeed, the argument above works as long as $p^k\le\norm{\mu}_2^{-2m}$ (approximately), and since $m\approx\frac{1}{6}n$, this means that we get a contribution of $\norm{\mu}_2^{-n/3}$ to $\disc(R)$ for every prime $p$. Two additional factors comes from the facts
that, for various reasons, we can only work with primes $p\le\exp(\sqrt{n})$ (so $\sum_p \frac{1}{p}=(\frac{1}{2}+o(1))\log n$) and from the fact that $|\disc(R)|\le\exp((2+o(1))n\log n)$.

We finally remark that the Breuillard--Varj\'{u} techniques are only effective when the multiplicative order of $\alpha\in\FF_{p^k}$ is not too small. This is expected: the value of $R(1)$ is essentially a simple random walk on $\FF_p$, which mixes only after $\approx p^2$ steps. A similar phenomenon holds for $R(\alpha)$ if $\alpha$ is of small multiplicative order. We deal with this issue by applying the inverse Littlewood--Offord theorem, with which we show that the probability of having a ``generic'' double root in $\FF_p$ and a double root of small order in some $\FF_{p^k}$ is small.\medskip

The strcture of the paper is therefore as follows. In \Sref{sec:konyagin}, we prove the extended version of the Breuillard--Varj\'{u} argument. Using this we study in \Sref{sec:prime-divisors} the number of primes $p$ with a double root in $\FF_p$, and no other double root in $\FF_{p^k}$ for small~$k$. Issues
revolving around $a$ and $\alpha$ of small multiplicative order are handled in~\Sref{sec:nonadmissible}, while issues revolving around lifting of double zeros from $\FF_p$ to $\Z/p^2\Z$ are handled in~\Sref{sec:mod-p^2}. We explain the core argument sketched above in \Sref{sec:BSG}. Finally, we prove the main theorem in \Sref{sec:proof}, where we also give the necessary background from algebraic number theory about prime divisors of discriminants.

\subsection*{Notations}

Throughout the paper, $\mu$ is a non-degenerate finitely supported probability measure on $\Z$. We write $H_{\mu}\coloneqq\max\set{\left|u\right|\suchthat u\in\supp(\mu)}$.

For any prime $p$, we write $\FFF_p=\FF_p^{\times}$, and for any $k\ge 2$ we write $\FFF_{p^k}$ for the set of elements of $\FF_{p^k}$ which do not lie in a proper subfield of $\FF_{p^k}$.

We denote constants by $c$ and $C$. Numbered constants are fixed and do not change their values throughout the paper. Unnumbered constants may change their value between occurrences. All constants might depend on the measure $\mu$, including constant implicit in Landau's $O$ or Vinogradov's $\ll$ notation. The notation $o$ will refer to convergence as $n\to\infty$, possibly not uniformly in $\mu$ but uniformly in other parameters ($p$, $k$, etc).

The letter $R$ always denotes the random polynomial
\[
R(x) = x^n + \omega_{n-1}x^{n-1} + \cdots + \omega_0
\]
of degree $n$. For any $d\ge 1$ and $l\ge 0$, we write
\[
R_{d,d+l}(x) = \omega_{d+l}x^{d+l} + \cdots + \omega_d x^d.
\]

\subsection*{Use of AI}

We used GPT-6 Astra for proofreading and suggestions on presentation, and to help find a short proof of \Lref{lem:finite-group-ILO}. We checked all suggestions incorporated into the paper and take full responsibility for its contents.

\subsection*{Acknowledgements}

We thank P\'{e}ter Varj\'{u} and Lior Bary-Soroker for many insightful discussions. %

GB is supported by the ERC consolidator grant CUTOFF (101123174). Views and opinions expressed are however those of the authors only and do not necessarily reflect those of the European Union or the European Research Council Executive Agency.

GK is supported by the Israel Science Foundation and by the Binational Science Foundation.

For the purpose of open access, the authors have applied a CC BY public copyright licence to any author accepted manuscript arising from this submission.

\section{Preliminaries}\label{sec:prelim}

\subsection{Resultants and discriminants}

We recall some useful bounds about resultants and discriminants that will be used in the paper. For a polynomial $P(x)=a_nx^n+a_{n-1}x^{n-1}+\cdots+a_0\in\Z[x]$ of degree $n$, we write $H(P)\coloneqq\max_{0\le i\le n}\left|a_i\right|$ for its height and $\norm{P}_2\coloneqq\left(\sum_{i=0}^n a_i^2\right)^{1/2}$.

\begin{fact}
For any two polynomials $P_1,P_2\in\Z[x]$ have degrees $m,n$ respectively, then we have
\[
|\Res(P_1,P_2)| \le \norm{P_1}_2^n \norm{P_2}^m \le  (m+1)^{n/2}(n+1)^{m/2}H(P_1)^nH(P_2)^m.
\]
\end{fact}

We recall that the Mahler measure of a polynomial $P(x)=a_n\prod_{i=1}^n(x-z_i)$ is given by $M(P)=a_n\prod_{i=1}^n\max\set{|z_i|,1}$. Mahler's bound on the discriminant \cite[Theorem 1]{Mahler64} yields:

\begin{fact}\label{fact:disc-bound}
For any polynomial $P(x)\in\Z[x]$ of degree~$n$, the Mahler measure of $P$ satisfies
\[
M(P) \le \norm{P}_2 \le \sqrt{n+1}H(P).
\]
Therefore, writing the discriminant as a determinant and using Hadamard's bound gives
\[
|\disc(P)| \le n^n M(P)^{2n-2} \le n^n(\sqrt{n+1}H(P))^{2n-2} \le (nH(P))^{2n}.
\]
\end{fact}

\subsection{Admissible elements}\label{sec:admiss}

An element of a finite field $\FF_{p^k}$ is called \textbf{admissible} if its multiplicative order is at least $n$ (and, if $k\ge 2$, it lies in $\FFF_{p^k}$). We write $\AA_{p^k}$ for the set of admissible elements in $\FF_{p^k}$, and $\BB_{p^k}$ for the set of elements of $\FFF_{p^k}$ which have multiplicative order smaller than $n$.

\begin{lem}\label{lem:count-inadmissible}
Let $p$ be a prime. Then
\[
\sum_{k=1}^{\infty}\left|\BB_{p^k}\right| \le n^2.
\]
\end{lem}

\begin{proof}
By definition, the sets $\set{\BB_{p^k}}_{k=1}^{\infty}$ are pairwise disjoint, and their union is precisely the set of nonzero elements of $\overline{\FF_p}$ of multiplicative order less than $n$. For each $1\le d\le n$ there are at most $\varphi(d)$ elements in $\overline{\FF_p}$ of order $d$, hence
\[
\sum_{k=1}^{\infty}\left|\BB_{p^k}\right| = \left|\bigcupdot_{k=1}^{\infty}\BB_{p^k}\right| \le \sum_{d=1}^{n-1} \varphi(d) \le n^2. \qedhere
\]
\end{proof}

\section{Mixing estimates for a random polynomial and its derivatives}\label{sec:konyagin}

In this section we will study the mixing time of the random polynomial and its derivatives in several finite fields. This will allow us to estimate the number of primes $p$ for which the random polynomial $R$ has a unique double root in $\FF_p$, while not having a triple root in $\FF_p$. We will also use the results of this section to bound the probability that $R$ has a double root in small extensions of $\FF_p$.

As in previous works, we cannot expect the mixing results to be the same for all $\alpha\in\FF_{p^k}$. Indeed, if the multiplicative order of $\alpha$ is small, then $R(\alpha)$ mixes much slower than for generic $\alpha\in\FF_{p^k}$. We will therefore be mostly interested in the case where $\alpha$ is admissible, and add this assumption whenever necessary.

\subsection{Setup and notation}

We first set up the notations which will be used throughout this section. We fix $M$ distinct primes $p_1,\dots,p_M$. For each $1\le j\le M$, we take $t_j$ (not necessarily distinct) prime powers $q_{j,1}=p_j^{k_{j,1}},\dots,q_{j,t_j}=p_j^{k_{j,t_j}}$. We set
\[
\mathcal{V}_j = \bigoplus_{l=1}^{t_j} \FF_{q_{j,l}}
\]
for each $1\le j\le M$, and
\begin{align*}
\mathcal{V} &= \bigoplus_{j=1}^M\mathcal{V}_j, & T &= \max_{1\le j\le M}t_j & D &= \max_{1\le j\le M}\sum_{l=1}^{t_j} k_{j,l}, & Q = \prod_{j=1}^M p_j.
\end{align*}

For any $\alpha=(\alpha_1,\dots,\alpha_M)\in\mathcal{V}$, we write $\alpha_j=(\alpha_{j,1},\dots,\alpha_{j,t_j})\in\mathcal{V}_j$. We also write $\alpha^m=(\alpha_{j,l}^m)_{j,l}$, $m\alpha=(m\alpha_{j,l})_{j,l}$, and for another $\beta\in\mathcal{V}$ we write $\alpha\beta=(\alpha_{j,l}\beta_{j,l})_{j,l}$.

For $\alpha\in\mathcal{V}$, we write
\[
\tr(\alpha) = (\tr(\alpha_1),\dots,\tr(\alpha_M)) \in \bigoplus_{j=1}^M\FF_{p_j},
\]
where
\[
\tr(\alpha_j) = \sum_{l=1}^{t_j}\tr_{\FF_{q_{j,l}}/\FF_{p_j}}(\alpha_{j,l}).
\]
For each $1\le j\le M$ we fix the standard additive identification $\iota_j\colon \FF_{p_j}\to \Z/p_j\Z$. We define an additive homomorphism $\psi_j\colon\FF_{p_j}\to\Z/Q\Z$ by
\[
\psi_j(x)=\frac{Q}{p_j}\iota_j(x).
\]
We can then define $\Psi\colon\bigoplus_{j=1}^M\FF_{p_j}\to\Z/Q\Z$ by
\begin{equation}\label{eq:def Psi}
\Psi(\alpha)=\sum_{j=1}^M\psi_j(\alpha_j).
\end{equation}

\begin{defn}\label{def:generic}
We say that $\alpha\in\mathcal{V}$ is \textbf{generic} if, for every $1\le j\le M$, the coordinates $\alpha_{j,1},\dots,\alpha_{j,t_j}$ are nonzero, no two of them are Galois conjugate over~$\FF_{p_j}$, and each $\alpha_{j,l}$ does not lie in a proper subfield of $\FF_{q_{j,l}}$.
\end{defn}

We fix a non-degenerate finitely supported probability measure $\mu$ on $\Z$. We then take a random polynomial
\[
R(x) = x^n + \omega_{n-1}x^{n-1} + \cdots + \omega_0\in\Z[x],
\]
where $\omega_1,\omega_2,\dots$ are i.i.d.\ $\mu$-random variables, and $\omega_0$ is independent of $\omega_1,\omega_2,\dots$. We do not assume that $\omega_0$ has the same distribution as the other coefficients (this will be convenient to handle the root 0). We also fix a positive integer $r\ge 1$, and assume throughout that
\[
r<\min\set{p_1,\dots,p_M}.
\]
For a given $\alpha\in\mathcal{V}$, we define $\nu_{\alpha}^{(n)}$ to be the law of~$(R(\alpha),R'(\alpha),\dots,R^{(r-1)}(\alpha))\in\mathcal{V}^r$. Furthermore, for nonnegative integers $d,l\ge 0$, we write $\nu_{\alpha}^{(d,d+l)}$ for the law of $(R_{d,d+l}(\alpha),R_{d,d+l}'(\alpha),\dots,R_{d,d+l}^{(r-1)}(\alpha))$, where
\[
R_{d,d+l}(x)=\omega_{d+l}x^{d+l}+\cdots+\omega_dx^d.
\]
For any $j\ge 0$, we set
\[
\mathbf{v}_j(\alpha) \coloneqq \left(\alpha^j, \df{j}{1}\alpha^{j-1}, \dots, \df{j}{r-1}\alpha^{j-(r-1)}\right) \in \mathcal{V}^r
\]
where $\df{j}{k}=j(j-1)\cdots(j-(k-1))$, so that
\begin{equation}\label{eq:expr-vj}
  (R_{d,d+l}(\alpha),R_{d,d+l}'(\alpha),\dots,R_{d,d+l}^{(r-1)}(\alpha)) = \sum_{j=d}^{d+l}\omega_j\mathbf{v}_j(\alpha).
\end{equation}
We interpret the $s$-th entry of $\mathbf{v}_j(\alpha)$ as zero if $s>j$.

\begin{lem}\label{lem:linear-transform}
For any $\alpha\in\mathcal{V}$ and integers $d\ge 1$ and $l,t\ge 0$, we have
\[
\nu_{\alpha}^{(d+t,d+l+t)} = (A_{\alpha}^t)_*\nu_{\alpha}^{(d,d+l)},
\]
where $A_{\alpha}$ is the linear operator defined on $\mathcal{V}^r$ by
\[
(A_{\alpha}v)^{(s)} = \alpha v^{(s)} + sv^{(s-1)}
\]
for every $v=(v^{(0)},\dots,v^{(r-1)})\in\mathcal{V}^r$ and $0\le s\le r-1$ (where $v^{(-1)}=0$).
\end{lem}

\begin{proof}
For the operator $A_{\alpha}$ defined above, we have
\begin{align*}
  (A_{\alpha}\mathbf{v}_j(\alpha))^{(s)} &= \alpha(\mathbf{v}_j(\alpha))^{(s)} + s(\mathbf{v}_j(\alpha))^{(s-1)} = \df{j}{s}\alpha^{j-s+1} + s\df{j}{s-1}\alpha^{j-(s-1)} \\
  &= \df{j+1}{s}\alpha^{j+1-s} = (\mathbf{v}_{j+1}(\alpha))^{(s)}
\end{align*}
for every $j\ge 0$ and $0\le s\le r-1$. Therefore $A_{\alpha}\mathbf{v}_j(\alpha)=\mathbf{v}_{j+1}(\alpha)$, so by induction we have $A_{\alpha}^t\mathbf{v}_j(\alpha)=\mathbf{v}_{j+t}(\alpha)$ for any $t\ge 0$. Since $\omega_{d+t},\dots,\omega_{d+t+l}$ have the same joint law as $\omega_d,\dots,\omega_{d+l}$, then $\sum_{j=d+t}^{d+l+t}\omega_j\mathbf{v}_j(\alpha) = A_{\alpha}^t\sum_{j=d}^{d+l}\omega_{j+t}\mathbf{v}_j(\alpha)$ has the same law as $A_{\alpha}^t\sum_{j=d}^{d+l}\omega_j\mathbf{v}_j(\alpha)$, so the claim follows from \eqref{eq:expr-vj}.
\end{proof}

\subsection{Fourier coefficients}

We will now give a bound on the Fourier coefficients of~$\nu_{\alpha}^{(d,d+l)}$, which are indexed by $\mathcal{V}^r$. For $\xi\in\mathcal{V}^r$, we write $\xi_{j,l}=(\xi_{j,l}^{(0)},\dots,\xi_{j,l}^{(r-1)})\in \FF_{q_{j,l}}^r$ for its $(j,l)$-coordinate. In particular, $\xi_{j,l}\ne 0$ means that at least one of the coordinates $\xi_{j,l}^{(0)},\dots,\xi_{j,l}^{(r-1)}$ is nonzero.

Write $e_Q(x)=\exp(2\pi ix/Q)$ for every $x\in\Z/Q\Z$. It is well known and easy to see that the additive characters of $\mathcal{V}^r$ are given by
\[
\chi_{\xi}(v) = e_Q\left(\Psi\circ\tr(\left\langle\xi,v\right\rangle)\right)
\]
for any $\xi,v\in\mathcal{V}^r$, where
\[
\left\langle\xi,v\right\rangle \coloneqq \left(\sum_{t=0}^{r-1}\xi_{j,l}^{(t)}v_{j,l}^{(t)}\right)_{j,l}\in\mathcal{V}.
\]
and $\Psi$ is from (\ref{eq:def Psi}).

For each $v\in\mathcal{V}^r$, we write $\mu.\delta_v$ for the probability measure on $\mathcal{V}^r$ defined by $\sum_{a\in\Z}\mu(a)\delta_{av}$. Also, for $a\in\Z/Q\Z$, we write $[a]^{\sim}$ for its unique representative in $(-\frac{Q}{2},\frac{Q}{2}]$.

\begin{lem}\label{lem:fourier-mu-trans}
For any $v,\xi\in\mathcal{V}^r$,
\[
\left|\widehat{\mu.\delta_v}(\xi)\right| \le \exp\left(-\sum_{a_1,a_2\in\Z}\frac{\mu(a_1)\mu(a_2)([\Psi\circ\tr(\left\langle(a_1-a_2)\xi,v\right\rangle)]^{\sim})^2}{Q^2}
\right)
\]
\end{lem}

\begin{proof}
Since $\widehat{\mu.\delta_v}(\xi)=\sum\mu(s)\xi_\zeta(av)$ and since $\mu$ is real we get
\[
\left|\widehat{\mu.\delta_v}(\xi)\right|^2 = \sum_{a_1,a_2\in\Z}\mu(a_1)\mu(a_2)\chi_{\xi}((a_1-a_2)v).
\]
As $\left|\widehat{\mu.\delta_v}(\xi)\right|^2\in\R$ and $\Re(e_Q(a)) \le 1-2([a]^{\sim})^2/Q^2$ if $a\in\Z/Q\Z$, we may use the definition of $\chi_{\xi}$ and deduce
\begin{align*}
\left|\widehat{\mu.\delta_v}(\xi)\right|^2 & \le 1 - 2\sum_{a_1,a_2\in\Z}\frac{\mu(a_1)\mu(a_2)([\Psi\circ\tr(\left\langle\xi,(a_1-a_2)v\right\rangle)]^{\sim})^2}{Q^2} \\
&= 1 - 2\sum_{a_1,a_2\in\Z}\frac{\mu(a_1)\mu(a_2)([\Psi\circ\tr(\left\langle(a_1-a_2)\xi,v\right\rangle)]^{\sim})^2}{Q^2}.
\end{align*}
However, for any $t$ we have $1-t\le\exp(-t)$, so the claim follows.
\end{proof}

For $\xi\in\mathcal{V}^r$, write
\[
S_j(\xi) \coloneqq \Psi\circ\tr(\left\langle\xi,\mathbf{v}_j(\alpha)\right\rangle)\in\Z/Q\Z.
\]

\begin{cor}\label{cor:fourier-nu-bound}
For every $d\ge 1$, every $l\ge 0$, and every $\xi\in\mathcal{V}^r$,
\[
\left|\whnu{_{\alpha}^{(d,d+l)}}(\xi)\right| \le \exp\left(-\sum_{a_1,a_2\in\Z}\left(\mu(a_1)\mu(a_2)\sum_{j=d}^{d+l}\frac{([S_j((a_1-a_2)\xi)]^{\sim})^2}{Q^2}
\right)\right)
\]
\end{cor}
($\whnu{_{\alpha}^{(d,d+l)}}$ is the Fourier transform of $\nu_{\alpha}^{(d,d+l)}$, for graphical reasons we do not put the hat over the entire expression).
\begin{proof}
The Fourier transform of a convolution is the product of the Fourier transforms, so
\[
\left|\whnu{_{\alpha}^{(d,d+l)}}(\xi)\right| = \left|\prod_{j=d}^{d+l}\widehat{\mu.\delta_{\mathbf{v}_j(\alpha)}}(\xi)\right|.
\]
We now apply \Lref{lem:fourier-mu-trans} to get
\[
\left|\whnu{_{\alpha}^{(d,d+l)}}(\xi)\right| \le \exp\left(-\sum_{a_1,a_2\in\Z}\sum_{j=d}^{d+l}\frac{\mu(a_1)\mu(a_2)([S_j((a_1-a_2)\xi)]^{\sim})^2}{Q^2}
\right)
\]
proving the corollary.
\end{proof}

\subsection{Estimates for the Fourier coefficients}

Our next goal is to give a uniform bound on the factor
\[
\sum_{j=d}^{d+l}([S_j((a_1-a_2)\xi)]^{\sim})^2
\]
appearing in \Cref{cor:fourier-nu-bound} for generic $\alpha\in\mathcal{V}$. For ease of notation, we think of $a_1,a_2,\xi$ as fixed, and write $\beta=(a_1-a_2)\xi$. We further write $S_j=S_j(\beta)$ and $\widetilde{S_j}=[S_j]^{\sim}$.

The proof uses a method of Konyagin \cite{Konyagin92} (see also \cite[Section 5]{BreuillardVarju19}). However, since our finite fields are not necessarily prime, and since we also include the derivatives of the polynomial, we need the following generalization of \cite[Lemma 29]{BreuillardVarju19}:

\begin{lem}\label{lem:apply-pol}
  Let $p>r$ be a prime, let $q_1=p^{k_1},\dots,q_t=p^{k_t}$ be powers of $p$, and for each $1\le l\le t$ let $\alpha_l\in\FF_{p^{k_l}}$ be generic (in the sense of Definition \ref{def:generic}). %
  For every $1\le l\le t$, let $0\ne\beta_l=(\beta_l^{(0)},\dots,\beta_l^{(r-1)})\in\FF_{q_l}^r$.

Let $P(x)=b_0+\cdots+b_ex^e\in\Z[x]$ be a polynomial such that
\begin{equation}\label{eq:pol-apply}
  \sum_{i=0}^e b_i \sum_{l=1}^t \tr_{\FF_{q_l}/\FF_p}(\left\langle\beta_l,\mathbf{v}_{i+m}(\alpha_l)\right\rangle) = 0
\end{equation}
in $\FF_p$ for every $0\le m\le r\sum_{l=1}^tk_l - 1$. Then $P(\alpha_l)=0$ for every $1\le l\le t$, and \eqref{eq:pol-apply} holds for all $m\ge 0$.
\end{lem}

\begin{proof}
We first note that for any $m\ge 0$ and any $0\le s\le r-1$,
\begin{align*}
  \left(\sum_{i=0}^e b_i\mathbf{v}_{i+m}(\alpha_l)\right)_s &= \sum_{i=0}^e \df{i+m}{s} b_i\alpha_l^{i+m-s} \\
  &= \sum_{i=0}^e \sum_{u=0}^s \binom{s}{u}\df{i}{s-u}(m)_u b_i \alpha_l^{i+m-s} \\
  &= \sum_{u=0}^s \binom{s}{u}\df{m}{u} \alpha_l^{m-u} \sum_{i=0}^e  \df{i}{s-u}b_i\alpha_l^{i-(s-u)} \\
  &= \sum_{u=0}^s \binom{s}{u}\df{m}{u} \alpha_l^{m-u} P^{(s-u)}(\alpha_l).
\end{align*}
For each $s\in\{0,\dotsc,r-1\}$ we use this to rewrite \eqref{eq:pol-apply} and sum over $s$ to get
\[
\sum_{l=1}^t \sum_{s=0}^{r-1} \sum_{u=0}^s \binom{s}{u}\df{m}{u} \tr_{\FF_{q_l}/\FF_p}\left(\beta_l^{(s)}\alpha_l^{m-u} P^{(s-u)}(\alpha_l)\right) = 0,
\]
namely
\[
\sum_{l=1}^t \sum_{s=0}^{r-1} \sum_{u=0}^s \sum_{j=0}^{k_l-1} \binom{s}{u}\df{m}{u} (\beta_l^{(s)})^{p^j}\alpha_l^{p^j(m-u)} P^{(s-u)}(\alpha_l^{p^j}) = 0.
\]
Changing the order of summation of the last equation, we may rewrite it as
\begin{equation}\label{eq:pol-apply2}
  \sum_{l=1}^t  \sum_{u=0}^{r-1} \sum_{j=0}^{k_l-1} \df{m}{u} \alpha_l^{p^j(m-u)} x_{l,u,j}= 0
\end{equation}
where $x_{l,u,j} = \sum_{s=u}^{r-1} \binom{s}{u} (\beta_l^{(s)})^{p^j} P^{(s-u)}(\alpha_l^{p^j})$ for all $l,u,j$.

We think of $\set{x_{l,u,j}}$ as variables, and put the coefficients of \eqref{eq:pol-apply2} for all $m=0,\dots,r\sum_{l=1}^t k_l-1$ in the columns of a square matrix $A$. Then $A$ is the confluent Vandermonde matrix corresponding to $\big\{\alpha_l^{p^j}\big\}$ of order $r$, and thus its determinant is a nonzero scalar multiple of $\prod_{(l,j)<(l',j')}\big(\alpha_l^{p^j} - \alpha_{l'}^{p^{j'}}\big)^{r^2}$. In particular, the matrix is non-singular thanks to the assumptions on $\alpha_1,\dots,\alpha_t$ (in particular, the fact that $\alpha$ is generic). This forces $x_{l,u,j}=0$ for all $l,u,j$.

Fix $1\le l\le t$, and let $u$ be the maximal index so that $\beta_l^{(u)}\ne 0$. Then $x_{l,u,0} = \beta_l^{(u)}P(\alpha_l)=0$, so $P(\alpha_l)=0$. Moreover, \eqref{eq:pol-apply} holds for all $m\ge 0$ since $x_{l,u,j}=0$ for all $l,u,j$, completing the proof.
\end{proof}

\begin{prop}\label{prop:Sj-tilde-bound}
There exist universal constants $\Ckon,\Cord>0$ such that the following holds. Let $\alpha\in\mathcal{V}$ and $\beta\in\mathcal{V}^r$, such that $\alpha$ is generic. %

Let $L\ge \Ckon\log Q^{rD}(\log\log Q^{rD})^4$ be an integer, and suppose that
\[
\sum_{j=0}^L\widetilde{S_j}^2 \le \frac{Q^2}{8\log(4L)}.
\]
Then, for every $j,l$ such that $\beta_{j,l}\ne 0$, the element $\alpha_{j,l}$ has multiplicative order smaller than $\Cord\log Q^{rD}\log\log\log Q^{rD}$.
\end{prop}

Before the proof, we recall some notations and facts from \cite{Konyagin92,BreuillardVarju19}. Let $N\ge e\ge 0$ be integers, and let $X=(x_0,\dots,x_N)$ be a sequence of integers. We write $\Lambda_e(X)$ for the `approximate ideal of $X$' i.e.\ the set of polynomials $P(x)=a_0+a_1x+\cdots+a_ex^e\in\Z[x]$ of degree at most $e$ such that
\[
a_0x_j + \cdots + a_ex_{j+e} = 0 \qquad \text{ for all } 0\le j\le N-e.
\]
We further denote by $\Lambda(X)$ the set of polynomials $P\in\Z[x]$ of degree at most $N$ such that $P\in\Lambda_{\deg P}(X)$. The following lemma is \cite[Corollary 28]{BreuillardVarju19}.

\begin{lem}\label{lem:approx-ideal} If $\Lambda(X)$ contains a nonzero polynomial of degree at most $\frac{N}{2}$, then there is a unique (up to multiplication by $\pm 1$) polynomial $P_0$ in $\Lambda(X)$ of minimal degree and with relatively prime coefficients. In this case, any polynomial in $\Lambda(X)$ of degree at most $N-\deg P_0$ must be divisible by $P_0$.
\end{lem}

\begin{proof}[Proof of Proposition \ref{prop:Sj-tilde-bound}]
We set $E=3\floor{\log Q^{rD}}$, so that, after increasing $\Ckon$ if necessary,
\[
1\le E\le \frac{1}{3}L,\qquad 2^E>Q^{rD},\qquad\text{and}\qquad \floor{L/6E}\ge 4\log E.
\]
We will first show that there exists a nonzero polynomial $P_1\in\Lambda\big(\{\widetilde{S_j}\}_{j=0}^L\big)$ of degree at most $E$.

We consider the set
\[
\Omega\coloneqq\set{(\sigma_0,\dots,\sigma_E)\in\set{\pm 1}^{E+1} \suchthat \left|\sum_{j=0}^E\sigma_j\widetilde{S_{j+m}}\right| < \frac{Q}{2}\text{ for every }0\le m\le L-E}.
\]
Let $\sigma_0,\dots,\sigma_E$ be independent, unbiased $\pm 1$-valued random variables. By Hoeffding's inequality and our assumption, for any $0\le m\le L-E$ we have
\[
\Pr\left(\left|\sum_{j=0}^E\sigma_j\widetilde{S_{j+m}}\right| \ge \frac{Q}{2}\right) \le 2\exp\left(-\frac{(Q/2)^2}{2\sum_{j=0}^E\widetilde{S_{j+m}}^2}\right) \le \frac{1}{2L}.
\]
Therefore $|\Omega|\ge \frac{1}{2}2^{E+1}>Q^{rD}$. By the pigeonhole principle, there exist distinct $x,y\in\Omega$ such that
\[
\sum_{j=0}^E x_j\widetilde{S_{j+m}} = \sum_{j=0}^E y_j\widetilde{S_{j+m}}
\]
for all $0\le m\le rD-1$. Put $a_j=\frac{x_j-y_j}{2}\in\set{0,\pm 1}$. Then the coefficients $a_j$ are not all zero, and
\begin{equation}\label{eq:sum-aj-Sj}
\sum_{j=0}^E a_j\widetilde{S_{j+m}}=0
\end{equation}
for all $0\le m\le rD-1$.

We now apply Lemma~\ref{lem:apply-pol} separately for each prime $p_j$, using only those coordinates $l$ for which $\beta_{j,l}\ne 0$. Since $\alpha$ is generic, the hypotheses of \Lref{lem:apply-pol} hold for these primes. As for the hypothesis \eqref{eq:pol-apply}, we take \eqref{eq:sum-aj-Sj} modulo $Q$ (i.e.\ remove the tildes) and use the fact that the $p_j$ are distinct, hence $\Psi$ is one-to-one. This shows that \eqref{eq:pol-apply} holds and we conclude from Lemma~\ref{lem:apply-pol} that
\[
\sum_{j=0}^E a_jS_{j+m}=0
\]
in $\Z/Q\Z$ for every $m\ge 0$.

Since $x,y\in\Omega$, we can return the tildes (by the definition of $\Omega$, $|\sum a_j\tilde{S}_{j+m}|<Q/2$ and is congruent to $\sum a_jS_{j+m}$ modulo $Q$). We get
\[
\sum_{j=0}^E a_j\widetilde{S_{j+m}}=0 %
\]
for every $0\le m\le L-E$. Thus $P_1(x)\coloneqq a_0+a_1x+\cdots+a_Ex^E$ belongs to
\[
\Lambda_E(\{\widetilde{S_j}\}_{j=0}^L) \sub \Lambda(\{\widetilde{S_j}\}_{j=0}^{\ceil{2L/3}})
\]
(the inclusion holds because $E\le L/3$).

Since $\ceil{2L/3}\ge 2E$, Lemma~\ref{lem:approx-ideal} %
gives a unique, up to sign, polynomial $P_0\in\Lambda(\{\widetilde{S_j}\}_{j=0}^{\ceil{2L/3}})$ of minimal degree and with relatively prime coefficients. It follows that $P_0\mid P_1$, so $\deg P_0\le E$.

Returning to our application of Lemma~\ref{lem:apply-pol}, we will now use the first conclusion of the lemma. We get, for each prime, that %
$P_0(\alpha_{j,l})=0$ whenever $\beta_{j,l}\ne 0$. Thus, for each such pair $(j,l)$, there is an irreducible factor $P_{j,l}$ of $P_0$ such that $P_{j,l}(\alpha_{j,l})=0$.

We now bound the Mahler measure of each $P_{j,l}$. Let $\zeta_1,\dots,\zeta_{\deg P_{j,l}}$ be the complex roots of $P_{j,l}$. Write $s=\floor{L/6E}$, so that $s\ge 4\log E$. By \cite[Lemma 26]{BreuillardVarju19}, there is a prime $\ell\in(s,2s]$ such that $\zeta_a/\zeta_b$ is not an $\ell$-th root of unity for any $a\ne b$. Hence the numbers $\zeta_1^\ell,\dots,\zeta_{\deg P_{j,l}}^\ell$ are all distinct.

Using the same pigeonhole argument as above, we can find a nonzero polynomial
\[
P_2(x)=\widetilde{P}_2(x^\ell), \qquad \widetilde{P}_2(x)=b_0+b_1x+\cdots+b_Ex^E,
\]
with $b_j\in\set{0,\pm 1}$, such that
\[
P_2\in \Lambda_{E\ell}(\{\widetilde{S_j}\}_{j=0}^L) \sub \Lambda(\{\widetilde{S_j}\}_{j=0}^{\ceil{2L/3}}).
\]
Here the inclusion follows from $E\ell\le L/3$. By the divisibility property of $P_0$, we have $P_0\mid P_2$. In particular, $P_{j,l}\mid P_2$. We get %
\[
M(P_{j,l})^\ell \stackrel{(*)}{\le} M(\widetilde{P}_2)
\stackrel{(**)}{\le} (E+1)^{1/2},
\]
where $(*)$ follows because the roots $\zeta_a^\ell$ are all distinct and $(**)$ follows from bounding the Mahler measure by the $L^2$ norm of the coefficients.
Therefore
\[
  M(P_{j,l}) \le (E+1)^{1/2\ell} \le (E+1)^{3E/L} \le \exp\left(\frac{30}{\Ckon(\log\log Q^{rD})^3}\right).
\]

We recall Dobrowolski's theorem~\cite{Dobrowolski79}: there is a universal constant $c>0$ such that, if $P\in\Z[x]$ is a polynomial with
\[
M(P) < \exp\left(c\left(\frac{\log\log\deg P}{\log\deg P}\right)^3\right),
\]
then $P$ is a product of powers of $x$ and cyclotomic polynomials. In our case,
\[
M(P_{j,l}) \le \exp\left(\frac{30}{\Ckon(\log\deg P_{j,l})^3}\right).
\]
Thus, after increasing $\Ckon$ if necessary (without dependence on any of the parameters), every such irreducible polynomial $P_{j,l}$ is cyclotomic or equal to $x$. The latter is impossible, because $P_{j,l}(\alpha_{j,l})=0$ while $\alpha_{j,l}\ne 0$, so $P_{j,l}$ must be cyclotomic. Since the degree of the $m$-th cyclotomic polynomial is Euler's totient function and $\varphi(m)\gg \frac{m}{\log\log m}$, we get
\[
\frac{m}{\log\log m}\ll \phi(m)=\deg P\le E,
\]
or equivalently $m\ll E\log\log E$. It follows that every $\alpha_{j,l}$ with $\beta_{j,l}\ne 0$ has multiplicative order smaller than $\Cord\log Q^{rD}\log\log\log Q^{rD}$. This proves the proposition.
\end{proof}

We can now conclude a uniform upper bound on the Fourier coefficients of $\nu_{\alpha}^{(d,d+l)}$:

\begin{cor}\label{cor:final-fourier-nu-bound}
Let $\alpha\in\mathcal{V}$ and $\xi\in\mathcal{V}^r\setminus\set{0}$, and assume that $\alpha$ is generic. Suppose that there exists a pair $(j_0,l_0)$ such that $\xi_{j_0,l_0}\ne 0$ and $\alpha_{j_0,l_0}$ has multiplicative order at least $\Cord\log Q^{rD}\log\log\log Q^{rD}$. Suppose further that $\supp(\mu)\sub\left(-\frac{1}{2}p_j,\frac{1}{2}p_j\right)$ for every $1\le j\le M$.

Let $L\ge \Ckon\log Q^{rD}(\log\log Q^{rD})^4$. Then
\[
\left|\whnu{_{\alpha}^{(d,d+l)}}(\xi)\right| < \exp\left(-\frac{1-\norm{\mu}_2^2}{8\log(4L)}\right)
\]
for all $d\ge 1$ and $l\ge L$.
\end{cor}

\begin{proof}
We recall from Corollary~\ref{cor:fourier-nu-bound} that
\[
\left|\whnu{_{\alpha}^{(d,d+l)}}(\xi)\right|
\le
\exp\left(-\sum_{a_1,a_2\in\Z}
\left(
\mu(a_1)\mu(a_2)
\sum_{j=d}^{d+l}
\frac{([S_j((a_1-a_2)\xi)]^{\sim})^2}{Q^2}
\right)\right).
\]
We first explain how to remove the shift by $d$. As in the proof of \Lref{lem:linear-transform}, we have
\[
\left\langle \xi,\mathbf{v}_{d+j}(\alpha)\right\rangle = \left\langle \xi,A_{\alpha}^d\mathbf{v}_j(\alpha)\right\rangle = \left\langle (A_{\alpha}^*)^d\xi,\mathbf{v}_j(\alpha)\right\rangle
\]
for every $j\ge 0$ and every $\xi\in\mathcal{V}^r$. Consequently,
\[
S_{d+j}(\xi)=S_j((A_{\alpha}^*)^d\xi)
\]
for every $j\ge 0$. Moreover, for every coordinate $(j,l)$, $\xi_{j,l}\ne 0$ if and only if $ ((A_{\alpha}^*)^d\xi)_{j,l}\ne 0$.

Now fix $a_1\ne a_2$ in the support of $\mu$. Since $\supp(\mu)\sub\left(-\frac{1}{2}p_j,\frac{1}{2}p_j\right)$ for every $j$, the integer $a_1-a_2$ is nonzero modulo every prime $p_j$. Hence $\big((a_1-a_2)(A_{\alpha}^*)^d\xi\big)_{j_0,l_0}\ne 0$. Applying Proposition~\ref{prop:Sj-tilde-bound} to $\beta=(a_1-a_2)(A_{\alpha}^*)^d\xi$ and using the assumed lower bound on the multiplicative order of $\alpha_{j_0,l_0}$, we obtain
\[
\sum_{j=0}^L \frac{([S_j((a_1-a_2)(A_{\alpha}^*)^d\xi)]^{\sim})^2}{Q^2} > \frac{1}{8\log(4L)}.
\]
Equivalently,
\[
\sum_{j=d}^{d+L} \frac{([S_j((a_1-a_2)\xi)]^{\sim})^2}{Q^2} > \frac{1}{8\log(4L)}.
\]
Since $l\ge L$, this gives
\[
\sum_{j=d}^{d+l} \frac{([S_j((a_1-a_2)\xi)]^{\sim})^2}{Q^2} > \frac{1}{8\log(4L)}.
\]

Returning to the Fourier bound, we get
\[
\left|\whnu{_{\alpha}^{(d,d+l)}}(\xi)\right| < \exp\left(-\frac{1}{8\log(4L)}\sum_{a_1\ne a_2}\mu(a_1)\mu(a_2)\right) = \exp\left(-\frac{1-\norm{\mu}_2^2}{8\log(4L)}\right).
\]
This proves the corollary.
\end{proof}

\begin{cor}\label{cor:iterated-fourier-nu-bound}
There is a universal constant $c>0$ such that the following holds. Let $\alpha\in\mathcal{V}$ and $\xi\in\mathcal{V}^r\setminus\set{0}$, and assume that $\alpha$ is generic. Suppose that there exists a pair $(j_0,l_0)$ such that $\xi_{j_0,l_0}\ne 0$ and $\alpha_{j_0,l_0}$ has multiplicative order at least $\Cord\log Q^{rD}\log\log\log Q^{rD}$. Suppose further that $\supp(\mu)\sub\left(-\frac{1}{2}p_j,\frac{1}{2}p_j\right)$ for every $1\le j\le M$. Then
\[
\left|\whnu{_\alpha^{(d,d+l)}}(\xi)\right| \le \exp\left(-\frac{cl}{\log Q^{rD}(\log\log Q^{rD})^5}\right)
\]
for all $d\ge 1$ and $l\ge \Ckon\log Q^{rD}(\log\log Q^{rD})^4$.
\end{cor}

\begin{proof}
We set $L \coloneqq \ceil{\Ckon\log Q^{rD}(\log\log Q^{rD})^4}$ and $K \coloneqq \floor{\frac{l+1}{L+1}}$. We divide the interval $\set{d,\dots,d+l}$ into $K$ consecutive intervals of length $L+1$, and discard the remaining indices, which we are allowed to do since $|\whnu{^{(\cdot,\cdot)}}|\le 1$ always. Since the coefficients are independent,
\[
\left|\whnu{_\alpha^{(d,d+l)}}(\xi)\right| \le \prod_{j=0}^{K-1}\left|\whnu{_\alpha^{
(d+j(L+1),\,d+j(L+1)+L)}}(\xi)\right|.
\]
By \Cref{cor:final-fourier-nu-bound}, each factor is at most $\exp\left(-\frac{1-\norm{\mu}_2^2}{8\log(4L)}\right)$. Therefore
\[
\left|\whnu{_\alpha^{(d,d+l)}}(\xi)\right| \le \exp\left(-\frac{K(1-\norm{\mu}_2^2)}{8\log(4L)}\right).
\]
After increasing $\Ckon$ if necessary, $K \gg \frac{l}{\log Q^{rD}(\log\log Q^{rD})^4}$, while $\log(4L) \ll \log\log Q^{rD}$. The result follows.
\end{proof}

We write $\AA_{\mathcal{V}}$ for the set of generic $\alpha\in\mathcal{V}$ such that each $\alpha_{j,l}$ is admissible (recall \S\ref{sec:admiss}). We can therefore deduce:

\begin{prop}\label{prop:full-mixing-bound}
There are absolute constants $C,c>0$ such that the following holds. Let $n\ge 1$ be a positive integer with
\begin{align*}
  n & \ge \frac{C}{1-\norm{\mu}_2^2}MT\log Q^{rD}(\log\log Q^{rD})^6, \\
  \log Q^{rD} & \ge \frac{1}{1-\norm{\mu}_2^2}.
\end{align*}
Suppose further that $\supp\mu\sub\left(-\frac{1}{2}p_j,\frac{1}{2}p_j\right)$ for all $1\le j\le M$. Then
\[
\left|\sum_{\alpha\in A}\nu_{\alpha}^{(n)}(0) - \frac{\left|A\right|}{\left|\mathcal{V}\right|^r}\right| \le \exp\left(-\frac{cn}{\log Q^{rD}(\log\log Q^{rD})^5}\right)
\]
for any nonempty $A\sub\AA_{\mathcal{V}}$.
\end{prop}

\begin{proof}
Fix $A\sub\AA_{\mathcal{V}}$. Since $\nu(0)=\frac{1}{|\mathcal{V}|^r}\sum_{\xi\in\mathcal{V}^r}\widehat{\nu}(\xi)$ and $\widehat{\nu}(0)=1$ for any probability measure $\nu$ on $\mathcal{V}^r$, we have
\begin{equation}\label{eq:kon-mixing-eq1}
\left|\sum_{\alpha\in A}\nu_{\alpha}^{(n)}(0)-\frac{|A|}{|\mathcal{V}|^r}\right| \le \frac{1}{|\mathcal{V}|^r}\sum_{\alpha\in A}\sum_{\xi\in\mathcal{V}^r\setminus\{0\}} \left|\whnu{_{\alpha}^{(n)}}(\xi)\right|.
\end{equation}

We begin by bounding
\[
\sum_{\alpha\in A}\|\nu_{\alpha}^{(d,d+l)}\|_2^2 \le \sum_{\alpha\in\mathcal{V}}\|\nu_{\alpha}^{(d,d+l)}\|_2^2.
\]
Let $\widetilde{R}_{d,d+l}$ be an independent copy of $R_{d,d+l}$. Then
\begin{align*}
  \sum_{\alpha\in\mathcal{V}}\|\nu_{\alpha}^{(d,d+l)}\|_2^2 &= \sum_{\alpha\in\mathcal{V}} \Pr\left(R_{d,d+l}^{(i)}(\alpha) = \widetilde{R}_{d,d+l}^{(i)}(\alpha)\text{ for every }0\le i\le r-1 \right)\\
  &\le \E\left|\set{\alpha\in\mathcal{V} \suchthat R_{d,d+l}(\alpha)=\widetilde{R}_{d,d+l}(\alpha)}\right|.
\end{align*}
If $R_{d,d+l}\ne\widetilde{R}_{d,d+l}$, then $R_{d,d+l}-\widetilde{R}_{d,d+l}$ is nonzero in each of the fields $\FF_{p_j}$ appearing in $\mathcal{V}$ (since every nonzero coefficient difference has absolute value smaller than $p_j$ for every $j$), so it has at most $l+1$ roots in each of these fields. Hence
\[
\left|\set{\alpha\in\mathcal{V} \suchthat R_{d,d+l}(\alpha)=\widetilde{R}_{d,d+l}(\alpha)}\right| \le (l+1)^{MT}.
\]
Moreover, $\Pr\left(R_{d,d+l}=\widetilde{R}_{d,d+l}\right) = \|\mu\|_2^{2(l+1)}$, and thus
\[
\sum_{\alpha\in\mathcal{V}}\|\nu_{\alpha}^{(d,d+l)}\|_2^2 \le (l+1)^{MT}+|\mathcal{V}|\|\mu\|_2^{2(l+1)} = (l+1)^{MT}+|\mathcal{V}|\exp\left(-(l+1)H_2(\mu)\right).
\]
Set $l_0\coloneqq\left\lceil\frac{\log|\mathcal{V}|}{H_2(\mu)}\right\rceil$. Applying the previous estimate to intervals of length $l_0$ and using Parseval we obtain
\[
\frac{1}{|\mathcal{V}|^r}\sum_{\alpha\in\mathcal{V}}\sum_{\xi\in\mathcal{V}^r} \left|\whnu{_{\alpha}^{(d,d+l_0-1)}}(\xi)\right|^2 = \sum_{\alpha\in\mathcal{V}} \|\nu_{\alpha}^{(d,d+l_0-1)}\|_2^2\le 2l_0^{MT}.
\]
Since $\nu_{\alpha}^{(1,2l_0)} = \nu_{\alpha}^{(1,l_0)} * \nu_{\alpha}^{(l_0+1,2l_0)}$, two applications of Cauchy--Schwarz give
\begin{align}
  \frac{1}{|\mathcal{V}|^r}\sum_{\alpha\in\mathcal{V}}\sum_{\xi\in\mathcal{V}^r} \left|\whnu{_{\alpha}^{(1,2l_0)}}(\xi)\right|
  &\le \left(\frac{1}{|\mathcal{V}|^r}\sum_{\alpha\in\mathcal{V}}\sum_{\xi\in\mathcal{V}^r} \left|\whnu{_{\alpha}^{(1,l_0)}}(\xi)\right|^2 \right)^{1/2}\notag\\
  &\qquad\cdot\left(\frac{1}{|\mathcal{V}|^r} \sum_{\alpha\in\mathcal{V}}\sum_{\xi\in\mathcal{V}^r} \left|\whnu{_{\alpha}^{(l_0+1,2l_0)}}(\xi)\right|^2\right)^{1/2}\notag\\
  &\le 2l_0^{MT}.
\label{eq:fourier-bound1}
\end{align}
Put $l_1=n-2l_0-2$, so that $\nu_{\alpha}^{(2l_0+1,n-1)} = \nu_{\alpha}^{(2l_0+1,\,2l_0+1+l_1)}$. After increasing $C$ if necessary, the assumptions imply that
\[
l_1\ge \frac{n}{2} \qquad\text{and}\qquad l_1\ge \Ckon\log Q^{rD}(\log\log Q^{rD})^4.
\]
Moreover, since $\alpha\in\AA_{\mathcal{V}}$, every coordinate of $\alpha$ has multiplicative order at least $n$. The lower bound on $n$ therefore implies that every coordinate has multiplicative order at least $\Cord\log Q^{rD}\log\log\log Q^{rD}$. By Corollary \ref{cor:iterated-fourier-nu-bound}, for every $\alpha\in A$ and $\xi\in\mathcal{V}^r\setminus\{0\}$,
\[
\left|\whnu{_{\alpha}^{(2l_0+1,n-1)}}(\xi)\right| \le \exp\left(-\frac{cn}{\log Q^{rD}(\log\log Q^{rD})^5}\right).
\]

The coefficients indexed by $\set{1,\dots,2l_0}$ and $\set{2l_0+1,\dots,n-1}$ are independent, while the remaining constant and leading coefficients contribute Fourier factors of absolute value at most $1$. Hence
\[
\left|\whnu{_{\alpha}^{(n)}}(\xi)\right|
\le \left|\whnu{_{\alpha}^{(1,2l_0)}}(\xi)\right| \left|\whnu{_{\alpha}^{(2l_0+1,n-1)}}(\xi)\right|.
\]
Using \eqref{eq:fourier-bound1}, we obtain
\[
\frac{1}{|\mathcal{V}|^r} \sum_{\alpha\in A} \sum_{\xi\in\mathcal{V}^r\setminus\{0\}} \left|\whnu{_{\alpha}^{(n)}}(\xi)\right| \le 2l_0^{MT}\exp\left(-\frac{cn}{\log Q^{rD}(\log\log Q^{rD})^5}\right).
\]

Finally, $l_0 \le 1+(1-\norm{\mu}_2^2)^{-1}\log|\mathcal{V}| \le 1+(1-\norm{\mu}_2^2)^{-1}\log Q^D$, and hence, using $\log Q^{rD}\ge(1-\norm{\mu}_2^2)^{-1}$, we have $\log(2l_0^{MT}) \ll MT\log\log Q^{rD}$. The assumed lower bound on~$n$ allows this factor to be absorbed into the exponential. Thus, after decreasing~$c$ if necessary,
\[
\frac{1}{|\mathcal{V}|^r} \sum_{\alpha\in A}\sum_{\xi\in\mathcal{V}^r\setminus\{0\}} \left|\whnu{_{\alpha}^{(n)}}(\xi)\right| \le \exp\left(-\frac{cn}{\log Q^{rD}(\log\log Q^{rD})^5}\right).
\]
The proposition now follows from \eqref{eq:kon-mixing-eq1}.
\end{proof}

\section{Prime divisors of the discriminant and double roots in small extensions}\label{sec:prime-divisors}

In this section we use the mixing estimates of the previous section to estimate the number of primes $p$ up to $\exp(\sqrt{n}/\log^3 n)$ for which the polynomial $R$ has a unique admissible double root in $\FF_p$. We then show that for most of the primes with a double root, $R$ has neither a triple in $\FF_p$ nor a double root in small extensions of $\FF_p$.

For a set of primes $\PP$, we write $S(\PP)\coloneqq\sum_{p\in\PP}\frac{1}{p}$ and $S_2(\PP)\coloneqq\sum_{p\in\PP}\frac{1}{p^2}$. We record the following fact:

\begin{rem}\label{rem:triv-P}
Let $\PP$ be a set of primes such that $\min\PP\ge c$ and $\max\PP\le C$. Then $\left|\PP\right|\le C$, $S(\PP) \le \log\log C + O(1)$, and $S_2(\PP)\le\sum_{p\ge c}\frac{1}{p^2}=\frac{1+o(1)}{c\log c}$ as $c\to\infty$.
\end{rem}

We will mostly add the following assumptions on $\PP$, ensuring that all our estimates hold:

\begin{assum}\label{assum:P-prop}
We fix a constant $B_0\ge 3$, depending only on $\mu$, which will be chosen below (see \Rref{rem:B0}). We will use the following condition on a finite nonempty set of primes~$\PP$:
\begin{enumerate}[label=(P)]
  \item\label{item:P1} For every $p\in\PP$,
  \[
  n^{B_0} \le p \le \exp\left(\frac{\sqrt{n}}{\log^3 n}\right).
  \]
\end{enumerate}
\end{assum}

\begin{rem}
We will use freely throughout the paper the following consequence of~\ref{item:P1}. For every $A,B\ge 0$ and $c>0$, uniformly over the sets $\PP$ satisfying \ref{item:P1},
\[
n^A\left|\PP\right|^B\exp\left(-c\frac{\sqrt{n}}{\log^2 n}\right) = \exp\left(-(c+o(1))\frac{\sqrt{n}}{\log^2 n}\right).
\]
In particular,
\[
n^A\left|\PP\right|^B\exp\left(-c\frac{\sqrt{n}}{\log^2 n}\right) = o(S_2(\PP))
\]
and hence also $o(S(\PP))$.

Indeed, $\left|\PP\right| \le \max\PP \le \exp\left(\frac{\sqrt{n}}{\log^3 n}\right)$, while $S_2(\PP) \ge (\max\PP)^{-2}$.
\end{rem}

The following is a direct corollary of \Pref{prop:full-mixing-bound}:

\begin{cor}\label{cor:mix-app1}
There is a constant $c=c(\mu)>0$ and $n_0=n_0(\mu)$ such that the following holds for all $n\ge n_0$. Let $\PP$ be a finite set of primes satisfying \ref{item:P1}. Then the following bounds hold for any $p,p'\in\PP$:
\begin{enumerate}
  \item For every $A\sub\AA_p$,
  \[
  \left|\sum_{\alpha\in A}\Pr(R(\alpha)=R'(\alpha)=0\text{ in }\FF_p) - \frac{|A|}{p^2}\right| \le \exp\left(-\frac{c\sqrt{n}}{\log^2 n}\right)
  \]

  \item For every $A\sub\set{(\alpha,\beta)\in\AA_p^2\suchthat \alpha\ne\beta}$,
  \[
  \left|\sum_{(\alpha,\beta)\in A}\Pr\left(\begin{array}{c} R(\alpha)=R'(\alpha)=0,\\ R(\beta)=R'(\beta)=0 \end{array} \text{ in }\FF_p\right) - \frac{|A|}{p^4}\right| \le \exp\left(-\frac{c\sqrt{n}}{\log^2 n}\right)
  \]

  \item If $p\ne p'$, then for every $A\sub\AA_p\times\AA_{p'}$,
  \[
  \left|\sum_{(\alpha,\beta)\in A}\Pr\left(\begin{array}{c} R(\alpha)=R'(\alpha)=0\text{ in }\FF_p,\\ R(\beta)=R'(\beta)=0\text{ in }\FF_{p'} \end{array}\right) - \frac{|A|}{p^2(p')^2}\right| \le \exp\left(-\frac{c\sqrt{n}}{\log^2 n}\right)
  \]

  \item For every $A\sub\AA_p$,
  \[
  \left|\sum_{\alpha\in A}\Pr(R(\alpha)=R'(\alpha)=R''(\alpha)=0\text{ in }\FF_p) - \frac{|A|}{p^3}\right| \le \exp\left(-\frac{c\sqrt{n}}{\log^2 n}\right)
  \]
\end{enumerate}
\end{cor}

\begin{proof}
These are special cases of \Pref{prop:full-mixing-bound}, taking:
\begin{enumerate}
  \item $\mathcal{V}=\FF_p$ and $r=2$; %
  \item $\mathcal{V}=\FF_p\times\FF_p$ and $r=2$; %
  \item $\mathcal{V}=\FF_p\times\FF_{p'}$ and $r=2$; %
  \item $\mathcal{V}=\FF_p$ and $r=3$. %
\end{enumerate}
When $n$ is sufficiently large, the assumptions of \Pref{prop:full-mixing-bound} hold by \ref{item:P1}. Indeed, in all cases
\[
Q^{rD} \le \exp\left(\frac{4\sqrt{n}}{\log^3 n}\right),
\]
and thus
\[
\log Q^{rD}(\log\log Q^{rD})^6 \le \frac{4\sqrt{n}}{\log^3 n}\log^6(4\sqrt{n}) \le 2\sqrt{n}\log^3(4n).
\]
Hence $n\ge C(1-\norm{\mu}_2^2)^{-1}\log Q^{rD}(\log\log Q^{rD})^6$ for the constant $C$ of \Pref{prop:full-mixing-bound} when $n$ is sufficiently large. We next note that
\[
\frac{n}{\log Q^{rD}(\log\log Q^{rD})^5} \ge \frac{n}{2\sqrt{n}\log^2(4n)} \gg \frac{\sqrt{n}}{\log^2 n},
\]
so the claim follows.
\end{proof}

For any polynomial $P\in\Z[x]$ and a set of primes $\PP$, let $\mathcal{U}_{\PP}(P)$ denote the number of primes in $\PP$ for which $P$ has a unique admissible double root in $\FF_p$.

\begin{prop}\label{prop:num-unique-double}
Let $0<\delta<\frac{1}{3}$, and let $\PP$ be a finite nonempty set of primes satisfying \ref{item:P1} and such that $p \ge \frac{n^2}{\delta}$ for all $p\in\PP$. Then
\[
\Pr\left(\mathcal{U}_{\PP}(R) < (1-3\delta)S(\PP)\right) \le \frac{K}{\delta^2 S(\PP)}
\]
for some universal constant $K>0$.
\end{prop}

\begin{proof}
For each $p\in\PP$, let $N_p$ denote the number of admissible double roots of~$R$ in $\FF_p$, and write $N\coloneqq\sum_{p\in\PP} N_p$. By \Cref{cor:mix-app1},
\begin{align*}
  \left|\E[N_p] - \frac{|\AA_p|}{p^2}\right| &\le p\exp\left(-\frac{c\sqrt{n}}{\log^2 n}\right) \le \exp\left(\frac{C\sqrt{n}}{\log^3 n}\right)\exp\left(-\frac{c\sqrt{n}}{\log^2 n}\right) \\
  &\le \exp\left(-\frac{c\sqrt{n}}{2\log^2 n}\right)
\end{align*}
for large enough $n$. By \Lref{lem:count-inadmissible}, we have $|\AA_p|=p-|\BB_p|-1\ge p-n^2\ge(1-\delta)p$ for every $p\in\PP$. Therefore,
\begin{align*}
  \E[N] &= \sum_{p\in\PP}\E[N_p] \ge \sum_{p\in\PP}\frac{|\AA_p|}{p^2} - \left|\PP\right|\exp\left(-\frac{c\sqrt{n}}{2\log^2 n}\right) \\
  &\ge (1-\delta)S(\PP) - \exp\left(\frac{C\sqrt{n}}{\log^3 n}\right)\exp\left(-\frac{c\sqrt{n}}{2\log^2 n}\right) = (1-\delta-o(1))S(\PP).
\end{align*}

We next bound the variance of $N$. For each $p\in\PP$, $N_p(N_p-1)$ counts the number of pairs $(\alpha,\beta)\in\AA_p^2$ with $\alpha\ne\beta$ such that $\alpha,\beta$ are double roots of $R$, so \Cref{cor:mix-app1} gives
\[
\E[N_p(N_p-1)] \le \frac{|\AA_p|^2}{p^4} + \exp\left(-\frac{c\sqrt{n}}{\log^2 n}\right)
\]
so
\[
\Var(N_p) \le \E[N_p^2] = \E[N_p] + \E[N_p(N_p-1)] \le \frac{1}{p} + \frac{1}{p^2} + 2\exp\left(-\frac{c\sqrt{n}}{\log^2 n}\right).
\]
The same corollary (clause (3)) shows that for distinct $p,p'\in\PP$,
\[
\Cov(N_p,N_{p'}) \le 3\exp\left(-\frac{c\sqrt{n}}{\log^2 n}\right).
\]
We thus have
\begin{align*}
  \Var(N) &\le \sum_{p\in\PP}\left(\frac{1}{p} + \frac{1}{p^2} + 2\exp\left(-\frac{c\sqrt{n}}{\log^2 n}\right)\right) + 3\left|\PP\right|^2\exp\left(-\frac{c\sqrt{n}}{\log^2 n}\right) \\
  &\le (2+o(1))S(\PP).
\end{align*}
By Chebyshev's inequality,
\begin{align*}
  \Pr(N \le (1-2\delta)S(\PP)) &\le \Pr(\left|N-\E[N]\right| \ge (\delta-o(1)) S(\PP)) \\
  &\le \frac{(2+o(1))S(\PP)}{((\delta-o(1)) S(\PP))^2} = \frac{2+o(1)}{\delta^2 S(\PP)}.
\end{align*}
Next, write $N'=\sum_{p\in\PP}N_p(N_p-1)$. By the above,
\[
\E[N'] \le \sum_{p\in\PP}\left(\frac{1}{p^2} + \exp\left(-\frac{c\sqrt{n}}{\log^2 n}\right)\right) = O(1),
\]
so by Markov's inequality we have
\[
\Pr(N' > \delta S(\PP)) \le \frac{O(1)}{\delta S(\PP)}.
\]

Finally, we claim that $\mathcal{U}_{\PP}(R) \ge N-N'$. Indeed, $N-N' = \sum_{p\in\PP}N_p(2-N_p)$; thus every $p$ with a unique admissible double root in $\FF_p$ contributes $1$ to this sum, whereas all other summands are nonpositive. Therefore
\[
\Pr\left(\mathcal{U}_{\PP}(R) < (1-3\delta)S(\PP)\right) \le \Pr(N < (1-2\delta) S(\PP)) + \Pr(N' > \delta S(\PP)) \le \frac{K}{\delta^2 S(\PP)}
\]
for some universal constant $K>0$.
\end{proof}

Next, we write $\mathcal{T}_{\PP}(R)$ for the number of primes $p\in\PP$ for which $R$ has an admissible triple root in $\FF_p$.

\begin{prop}\label{prop:no-triple-roots}
Let $\PP$ be a set of primes satisfying \ref{item:P1}. Then
\[
\Pr(\mathcal{T}_{\PP}(R) > 0) \le S_2(\PP)(1+o(1)).
\]
\end{prop}

\begin{proof}
For each $p\in\PP$, let $T_p$ denote the number of admissible triple roots of $R$ in $\FF_p$. By \Cref{cor:mix-app1},
\[
\E[T_p] \le \frac{|\AA_p|}{p^3} + \exp\left(-c\frac{(1-\norm{\mu}_2^2)\sqrt{n}}{\log^2 n}\right) \le \frac{1}{p^2} + \exp\left(-c\frac{(1-\norm{\mu}_2^2)\sqrt{n}}{\log^2 n}\right).
\]
Summing over $p$,
\begin{align*}
  \E[\mathcal{T}_{\PP}(R)] &= \sum_{p\in\PP}\E[\ind{\set{T_p\ge 1}}] \le \sum_{p\in\PP}\E[T_p] \\
  &\le \sum_{p\in\PP}\frac{1}{p^2} + \left|\PP\right|\exp\left(-c\frac{(1-\norm{\mu}_2^2)\sqrt{n}}{\log^2 n}\right) = S_2(\PP)(1 + o(1)).
\end{align*}
The proposition then follows from Markov's inequality.
\end{proof}

Finally, we will need the following.
\begin{defn}\label{def:Bad}For a prime $p$ and a positive integer $k\ge 1$, let $B_{p^k}^{\ad}$ denote the event that there is an admissible double root of $R$ in $\FF_{p^k}$.
\end{defn}

\begin{prop}\label{prop:no-admissible-double-roots-ext}
Let $\PP$ be a finite nonempty set of primes satisfying \ref{item:P1}, and let $\ksm\ge 2$ be a positive integer such that
\[
p^{\ksm}\le\exp\left(\frac{\sqrt{n}}{\log^3 n}\right)
\]
for every $p\in\PP$. Then
\[
\Pr\left(\exists p\in\PP\;\exists 2\le k\le \ksm\;\exists \alpha\in\AA_{p^k}:R(\alpha)=R'(\alpha)=0\right) \le 2S_2(\PP)(1+o(1)).
\]
\end{prop}

\begin{proof}
In a similar manner to \Cref{cor:mix-app1}, we can apply \Pref{prop:full-mixing-bound} and deduce that for any $p\in\PP$, any $2\le k\le \ksm$, and any $\alpha\in\AA_{p^k}$,
\[
\Pr(R(\alpha)=R'(\alpha)=0) \le \frac{1}{p^{2k}} + \exp\left(-\frac{c\sqrt{n}}{\log^2 n}\right).
\]
For each $p,k$ as above, let $N_{p^k}$ denote the number of admissible double roots of $R$ in $\FF_{p^k}$, and let $N=\sum_{p\in\PP}\sum_{k=2}^{\ksm}N_{p^k}$. Therefore
\begin{align*}
  \E[N] &= \sum_{p\in\PP}\sum_{k=2}^{\ksm}\sum_{\alpha\in\AA_{p^k}} \Pr(R(\alpha)=R'(\alpha)=0) \\
  &\le \sum_{p\in\PP}\sum_{k=2}^{\ksm}\sum_{\alpha\in\AA_{p^k}}\left(\frac{1}{p^{2k}} + \exp\left(-\frac{c\sqrt{n}}{\log^2 n}\right)\right) \\
  &\le \sum_{p\in\PP}\sum_{k=2}^{\ksm}\left(\frac{1}{p^k} + p^k\exp\left(-\frac{c\sqrt{n}}{\log^2 n}\right)\right) \\
  &\le \sum_{p\in\PP}\left(\frac{2}{p^2} + n\exp\left(\frac{C\sqrt{n}}{(\log n)^3}\right)\exp\left(-\frac{c\sqrt{n}}{\log^2 n}\right)\right) \\
  &\le \sum_{p\in\PP}\frac{2}{p^2} + n\exp\left(\frac{2C\sqrt{n}}{(\log n)^3}\right)\exp\left(-\frac{c\sqrt{n}}{\log^2 n}\right) = 2S_2(\PP)(1+o(1)).
\end{align*}
Therefore the result follows by Markov's inequality.
\end{proof}

\section{Double roots modulo primes squared}\label{sec:mod-p^2}

We will also need to show that, with high probability, an admissible double root of $R$ in $\FF_p$ cannot be lifted to a double root of $R$ modulo $p^2$. We identify every $a\in\FF_p$ with its representative in $\set{0,\dots,p-1}$. For a polynomial $P\in\Z[x]$ and a double root $a\in\FF_p$ of $P$, we say that $a$ lifts to a double root modulo $p^2$ if there exists $b\in\FF_p$ such that
\[
P(a+pb)=P'(a+pb)=0\pmod{p^2}.
\]

We begin with the following elementary observation.

\begin{lem}\label{lem:double-root-lift}
Let $P\in\Z[x]$ be a monic polynomial. Let $p$ be a prime, and let $a\in\FF_p$ be a double root of the reduction $\overline{P}$ of $P$ modulo $p$. If $a$ lifts to a root of $P$ modulo $p^2$, then $P(a_0)\equiv 0\pmod{p^2}$ for every lift $a_0$ of $a$ in $\Z/p^2\Z$. If moreover $a$ is not a triple root of $\overline{P}$, then there is a unique lift $a_0 \in \Z/p^2\Z$ with $P(a_0)\equiv P'(a_0) \equiv 0\pmod{p^2}$.
\end{lem}

\begin{proof}
For every $b\in\FF_p$, we have
\[
P(a+pb)\equiv P(a)+pbP'(a)\pmod{p^2}
\]
If $a$ is a double root of $P$ modulo $p$, we have $P'(a)\equiv 0\pmod p$, and therefore
\[
P(a+pb)\equiv P(a)\pmod{p^2}.
\]
This proves the first assertion.

Suppose now that $a$ is not a triple root of $\overline{P}$ and let $a_0\in \Z/p^2\Z$ be a lift of $a$ with $P(a_0)\equiv 0\pmod{p^2}$. Write $P'(a_0)=pu\pmod{p^2}$ for some $u\in\FF_p$. Note that
\[
P'(a_0+pb)\equiv P'(a_0)+pbP''(a_0)\pmod{p^2}.
\]
So $P'(a_0+pb)\equiv 0\pmod{p^2}$ if and only if $u+bP''(a)=0\pmod p$. Since $a$ is not a triple root, $P''(a)\not\equiv 0\pmod p$, so this equation has a unique solution $b\in\FF_p$.
\end{proof}

We next prove a mixing estimate for the pair
\[
\left(R(a)\pmod{p^2},R'(a)\pmod p\right).
\]
We use similar notations as in the previous sections. For $a\in\FF_p^\times$ and integers $d\ge 1$ and $l\ge 0$, let $\rho_a^{(d,d+l)}$ denote the law on $\mathcal{W}_p\coloneqq\Z/p^2\Z\times\FF_p$ of
\[
\left(R_{d,d+l}(a)\pmod{p^2},R_{d,d+l}'(a)\pmod p\right).
\]
We further write $\rho_a^{(n)}$ for the law of
\[
\left(R(a)\pmod{p^2},R'(a)\pmod p\right).
\]
We emphasize that $a$ is identified with its representative in $\set{0,\dots,p-1}$, and all calculations with $a$ in the remainder of the section are done modulo $p^2$.

We identify the dual group of $\mathcal{W}_p$ with $\mathcal{W}_p$ by associating to
$\xi=(\xi_1,\xi_2)\in\mathcal{W}_p$ the character
\[
\chi_\xi(x,y)=e_{p^2}(\xi_1x+p\xi_2y).
\]
For $j\ge0$, write
\[
S_j(\xi)\coloneqq \xi_1a^j+p\xi_2ja^{j-1}\in\Z/p^2\Z.
\]
Also, for $x\in\Z/p^2\Z$, let $[x]_{p^2}^{\sim}$ denote its unique representative in
$\left(-\frac{p^2}{2},\frac{p^2}{2}\right]$.

The same argument used in \Lref{lem:fourier-mu-trans} and \Cref{cor:fourier-nu-bound} shows that
\begin{equation}\label{eq:fourier-p2}
\left|\whrho{_a^{(d,d+l)}}(\xi)\right| \le \exp\left(-\sum_{u,v\in\Z}\mu(u)\mu(v) \sum_{j=d}^{d+l}\frac{\left([S_j((u-v)\xi)]_{p^2}^{\sim}\right)^2}{p^4}\right).
\end{equation}

The following is the analogue of \Pref{prop:Sj-tilde-bound} that we need.

\begin{prop}\label{prop:Sj-tilde-bound-p2}
Let $p$ be a sufficiently large prime, let $a\in\FF_p^\times$, and let $0\ne\xi=(\xi_1,\xi_2)\in\Z/p^2\Z\times\FF_p$. Let $L\ge \Ckon\log p^4(\log\log p^4)^4$, and suppose that
\[
\sum_{j=0}^L\left([S_j(\xi)]_{p^2}^{\sim}\right)^2 \le \frac{p^4}{8\log(4L)}.
\]
Then the multiplicative order of $a$ is at most $\Cord\log p^4\log\log\log p^4$.
\end{prop}

\begin{proof}
Suppose first that $p\mid\xi_1$, and write $\xi_1=p\overline{\xi}_1$ for some $\overline{\xi}_1\in\set{0,\dots,p-1}$. Then $S_j(\xi)\equiv p\left(\overline{\xi}_1a^j+\xi_2ja^{j-1}\right)\pmod{p^2}$. Since $(\overline{\xi}_1,\xi_2)\ne 0$, we may apply \Pref{prop:Sj-tilde-bound} with $\mathcal{V}=\FF_p$ and $r=2$. This shows that the multiplicative order of $a$ is at most $\Cord\log p^2\log\log\log p^2$, and hence the claimed bound.

Suppose now that $p\nmid\xi_1$. We explain the only modification needed in the proof of \Pref{prop:Sj-tilde-bound}. Let $P(x)=b_0+\cdots+b_ex^e\in\Z[x]$, and suppose that
\begin{equation}\label{eq:recurrence-p2-first}
\sum_{i=0}^eb_iS_{i+m}(\xi)\equiv 0\pmod{p^2}
\end{equation}
for $m=0,1$. A direct calculation gives
\begin{equation}\label{eq:recurrence-p2-formula}
\sum_{i=0}^eb_iS_{i+m}(\xi) \equiv a^m\left(\xi_1P(a)+p\xi_2\left(ma^{-1}P(a)+P'(a)\right)\right) \pmod{p^2}.
\end{equation}

If $\xi_2=0$, then \eqref{eq:recurrence-p2-first} with $m=0$ gives $P(a)\equiv 0\pmod{p^2}$, since $p\nmid\xi_1$. It follows immediately from \eqref{eq:recurrence-p2-formula} that \eqref{eq:recurrence-p2-first} holds for every $m\ge0$.

Suppose that $\xi_2\ne0$. Using \eqref{eq:recurrence-p2-formula} for $m=0,1$, multiplying the $m=0$ expression by $a$ and subtracting the two resulting equations, we obtain $p\xi_2a^{-1}P(a)\equiv 0\pmod{p^2}$. Therefore $P(a)=0\pmod p$. Write $P(a)=pA\pmod{p^2}$. The equation corresponding to $m=0$ now gives $\xi_1A+\xi_2P'(a)=0\pmod p$. For every $m\ge 0$, it follows that
\[
\xi_1P(a)+p\xi_2\left(ma^{-1}P(a)+P'(a)\right) \equiv p\left(\xi_1A+\xi_2P'(a)\right) \equiv  0\pmod{p^2}.
\]
Thus, in both cases, $P(a)\equiv 0\pmod p$ and \eqref{eq:recurrence-p2-first} holds for every $m\ge0$.

We may now repeat the proof of \Pref{prop:Sj-tilde-bound}. We take $E=3\floor{\log p^4}$. In the pigeonhole argument there are at most $p^4$ possible pairs of residues corresponding to $m=0,1$. The preceding calculation replaces the application of \Lref{lem:apply-pol}: it shows that the polynomial obtained from the pigeonhole argument vanishes at $a$ modulo $p$, and that the corresponding recurrence holds for every $m\ge0$. The rest of the proof, including the construction of the integer recurrence and the Mahler measure argument, is unchanged, with $Q^{rD}$ replaced by $p^4$. We conclude that the multiplicative order of $a$ is at most $\Cord\log p^4\log\log\log p^4$, as required.
\end{proof}

We can now deduce the required bound on the Fourier coefficients.

\begin{cor}\label{cor:fourier-p2}
There is an absolute constant $c>0$ such that the following holds. Let $p$ be a prime, let $a\in\FF_p^\times$, and suppose that the multiplicative order of $a$ is at least $\Cord\log p^4\log\log\log p^4$. Assume further that $\supp(\mu)\sub(-\frac{p}{2},\frac{p}{2})$.

Then, for every $0\ne\xi\in\mathcal{W}_p$, every $d\ge1$, and every $l\ge \Ckon\log p^4(\log\log p^4)^4$, we have
\[
\left|\whrho{_a^{(d,d+l)}}(\xi)\right| \le \exp\left(-c\frac{(1-\norm{\mu}_2^2)l}{\log p^4(\log\log p^4)^5}\right).
\]
\end{cor}

\begin{proof}
For every $d\ge 0$, define $\Theta_d\colon\mathcal{W}_p\to\mathcal{W}_p$ by
\[
\Theta_d(\xi_1,\xi_2) = \left(a^d\xi_1+pda^{d-1}\xi_2,\,a^d\xi_2\right).
\]
The first coordinate is taken modulo $p^2$, and the second coordinate is taken modulo~$p$. The map $\Theta_d$ is an automorphism of $\mathcal{W}_p$, and a direct calculation similar to \Lref{lem:linear-transform} gives
\begin{equation}\label{eq:shift-p2}
S_{d+j}(\xi)=S_j(\Theta_d(\xi))
\end{equation}
for every $j\ge0$.

Let $L=\ceil{\Ckon\log p^4(\log\log p^4)^4}$. Fix $u\ne v$ in the support of $\mu$. Since $\supp(\mu)\sub\left(-\frac{p}{2},\frac{p}{2}\right)$, the integer $u-v$ is nonzero modulo $p$. It follows that $\Theta_d((u-v)\xi)\ne 0$. By \Pref{prop:Sj-tilde-bound-p2} and the assumed lower bound on the multiplicative order of $a$,
\[
\sum_{j=0}^L \left([S_j(\Theta_d((u-v)\xi))]_{p^2}^{\sim}\right)^2 > \frac{p^4}{8\log(4L)}.
\]
Using \eqref{eq:shift-p2}, we obtain
\[
\sum_{j=d}^{d+L} \frac{\left([S_j((u-v)\xi)]_{p^2}^{\sim}\right)^2}{p^4} > \frac{1}{8\log(4L)}.
\]
Substituting this bound in \eqref{eq:fourier-p2} gives
\[
\left|\whrho{_a^{(d,d+L)}}(\xi)\right| < \exp\left(-\frac{1}{8\log(4L)}\sum_{u\ne v}\mu(u)\mu(v)\right) \le \exp\left(-\frac{1-\norm{\mu}_2^2}{8\log(4L)}\right).
\]

We now divide $\set{d,\dots,d+l}$ into consecutive intervals of length $L+1$ and discard the remaining indices. Using the independence of the coefficients, as in the proof of \Cref{cor:iterated-fourier-nu-bound}, we obtain
\begin{equation*}
\left|\whrho{_a^{(d,d+l)}}(\xi)\right| \le \exp\left(-\frac{cl}{\log p^4(\log\log p^4)^5}\right).
\qedhere
\end{equation*}
\end{proof}

The following is the mixed-modulus version of the first part of \Cref{cor:mix-app1}.

\begin{prop}\label{prop:mix-p2}
There is an absolute constant $c>0$ such that the following holds. Let $p\le\exp\left(\frac{\sqrt{n}}{\log^3 n}\right)$ be a prime satisfying $\supp(\mu)\sub\left(-\frac{p}{2},\frac{p}{2}\right)$. Then, for every nonempty $A\sub\AA_p$,
\[
\left|\sum_{a\in A}\Pr\left(R(a)=0\spmod{p^2},\,R'(a)=0\spmod p\right) - \frac{|A|}{p^3}\right| \le \exp\left(-\frac{c\sqrt{n}}{\log^2 n}\right).
\]
\end{prop}

\begin{proof}
We follow the proof of \Pref{prop:full-mixing-bound}. We first establish the required $L^2$-bound. Let $\widetilde{R}_{d,d+l}$ be an independent copy of $R_{d,d+l}$. Then
\begin{align*}
  \sum_{a\in\FF_p^\times}\norm{\rho_a^{(d,d+l)}}_2^2 &= \sum_{a\in\FF_p^\times}\Pr\left(\begin{array}{c} R_{d,d+l}(a)=\widetilde{R}_{d,d+l}(a)\spmod{p^2},\\
  R_{d,d+l}'(a)=\widetilde{R}_{d,d+l}'(a)\spmod p
  \end{array}\right) \\
  &\le \E\left|\set{a\in\FF_p^\times \suchthat R_{d,d+l}(a)=\widetilde{R}_{d,d+l}(a)\spmod p}\right|.
\end{align*}
If $R_{d,d+l}\ne\widetilde{R}_{d,d+l}$, then their difference remains nonzero modulo $p$, since $\supp(\mu)\sub\left(-\frac{p}{2},\frac{p}{2}\right)$. After dividing by $x^d$, this difference has degree at most $l$, and therefore has at most $l$ roots in $\FF_p^\times$. Since $\Pr\left(R_{d,d+l}=\widetilde{R}_{d,d+l}\right) = \norm{\mu}_2^{2(l+1)}$, we obtain
\begin{equation}\label{eq:L2-p2}
\sum_{a\in\FF_p^\times}\norm{\rho_a^{(d,d+l)}}_2^2 \le l+p\norm{\mu}_2^{2(l+1)}.
\end{equation}

Set $l_0\coloneqq \ceil{\frac{\log p}{H_2(\mu)}}$. By \eqref{eq:L2-p2} and Plancherel's identity,
\[
\frac{1}{p^3}\sum_{a\in\FF_p^\times}\sum_{\xi\in\mathcal{W}_p} \left|\whrho{_a^{(d,d+l_0-1)}}(\xi)\right|^2 = \sum_{a\in\FF_p^\times}\norm{\rho_a^{(d,d+l_0-1)}}_2^2 \le 2l_0.
\]
Since $\rho_a^{(1,2l_0)} = \rho_a^{(1,l_0)}*\rho_a^{(l_0+1,2l_0)}$, Cauchy--Schwarz gives
\begin{equation}\label{eq:L1-p2}
\frac{1}{p^3}\sum_{a\in\FF_p^\times}\sum_{\xi\in\mathcal{W}_p} \left|\whrho{_a^{(1,2l_0)}}(\xi)\right| \le 2l_0.
\end{equation}
If $a\in\AA_p$, then the multiplicative order of $a$ is at least $n$, and by the assumptions on $p$ and $n$ we have $n \ge \Cord\log p^4\log\log\log p^4$ and that $n-2l_0-2\ge n/2$. We may therefore apply \Cref{cor:fourier-p2} to obtain
\[
\left|\whrho{_a^{(2l_0+1,n-1)}}(\xi)\right| \le \exp\left(-\frac{cn}{\log p^4(\log\log p^4)^5}\right)
\]
for every $a\in A$ and every $0\ne\xi\in\mathcal{W}_p$.

The constant and leading coefficients contribute Fourier factors of absolute value at most $1$. Hence, by \eqref{eq:L1-p2},
\[
\frac{1}{p^3}\sum_{a\in A}\sum_{\xi\in\mathcal{W}_p\setminus\set{0}} \left|\whrho{_a^{(n)}}(\xi)\right| \le 2l_0\exp\left(-\frac{cn}{\log p^4(\log\log p^4)^5}\right) \le \exp\left(-\frac{c\sqrt{n}}{\log^2 n}\right)
\]
where the last inequality follows from the assumed upper bound on $p$, after decreasing $c$ if necessary. The proposition now follows by Fourier inversion on $\mathcal{W}_p$.
\end{proof}

We can now rule out admissible double roots which lift modulo $p^2$.

\begin{prop}\label{prop:no-double-root-lifts}
Let $\PP$ be a finite nonempty set of primes satisfying \ref{item:P1}. Then
\[
\Pr\left(\begin{array}{c} \exists p\in\PP\ \exists a\in\AA_p \text{ such that}\\
a\text{ lifts to a double root of }R\text{ modulo }p^2 \end{array}\right) \le S_2(\PP)(1+o(1)).
\]
\end{prop}

\begin{proof}
Fix $p\in\PP$. By \Lref{lem:double-root-lift}, if $a\in\AA_p$ lifts to a double root modulo $p^2$, then $R(a)\equiv 0\spmod{p^2}$ and $R'(a)\equiv 0\spmod{p}$. Therefore, by \Pref{prop:mix-p2},
\begin{multline*}
  \Pr\left(\exists a\in\AA_p\text{ which lifts to a double root modulo }p^2\right) \\
  \le \sum_{a\in\AA_p}\Pr\left(R(a)=0\spmod{p^2},\,R'(a)=0\spmod p\right) \\
  \le \frac{|\AA_p|}{p^3} + \exp\left(-\frac{c\sqrt{n}}{\log^2 n}\right) \le \frac{1}{p^2} + \exp\left(-\frac{c\sqrt{n}}{\log^2 n}\right).
\end{multline*}
Summing over $p\in\PP$, we obtain
\[
\Pr\left(\begin{array}{c} \exists p\in\PP\ \exists a\in\AA_p \text{ such that}\\ a\text{ lifts to a double root of }R\text{ modulo }p^2 \end{array}\right) \le S_2(\PP) + |\PP|\exp\left(-\frac{c\sqrt{n}}{\log^2 n}\right).
\]
Finally,
\[
S_2(\PP) \ge \frac{|\PP|}{(\max\PP)^2},
\]
and the assumed upper bound on $\max\PP$ shows that the second term is $o(S_2(\PP))$. This completes the proof.
\end{proof}

\section{Non-admissible double roots}\label{sec:nonadmissible}

The previous sections studied the number of admissible double roots of $R$ in small extensions of $\FF_p$. However, we also need to rule out the existence of inadmissible double roots. To do this, we will bound the probability that for a prime $p$ and a positive integer $k\ge 1$, there are $(a,\alpha)\in\FF_p\times\FF_{p^k}$ such that $a$ is an admissible double root of $R$, and $\alpha$ is a double root of $R$ which has multiplicative order less than $n$.

\begin{prop}\label{prop:double-Fp-Fq}
There are absolute constants $C,c>0$ such that the following holds. Let $p$ be a prime number, and let $k\ge 1$ be a positive integer. Assume that $n\ge C\log p^{2k+2}(\log\log p^{2k+2})^4$ and that $\supp(\mu)\sub(-\frac{p}{2},\frac{p}{2})$. Then, for any admissible $a\in\FF_p$, and for any $\alpha\in\FFF_{p^k}$ (which need not be admissible) such that $\alpha\ne a$, we have
\begin{multline*}
  \Pr(R(a)=R'(a)=R(\alpha)=R'(\alpha)=0) \\
  \le \exp\left(-\frac{cn}{\log p^{2k+2}(\log\log p^{2k+2})^5}\right) + \frac{1}{p^2}\Pr(R(\alpha)=0).
\end{multline*}
\end{prop}

\begin{proof}
Let $\nu=\nu_{a,\alpha}^{(n)}$. By Fourier inversion,
\begin{equation}\label{eq:nonad-fourier-inversion}
  \Pr(R(a)=R'(a)=R(\alpha)=R'(\alpha)=0) = \frac{1}{p^{2k+2}}\sum_{(\xi_1,\xi_2,\xi_3,\xi_4)\in\FF_p^2\times\FF_{p^k}^2} \widehat{\nu}(\xi_1,\xi_2,\xi_3,\xi_4).
\end{equation}
We divide this sum into two parts. When $(\xi_1,\xi_2)\ne(0,0)$, we can use the bound of \Cref{cor:iterated-fourier-nu-bound} and deduce
\[
\left|\widehat{\nu}(\xi_1,\xi_2,\xi_3,\xi_4)\right| \le \exp\left(-\frac{cn}{\log p^{2k+2}(\log\log p^{2k+2})^5}\right).
\]
For $(\xi_1,\xi_2)=(0,0)$, we note that
\[
\widehat{\nu}(0,0,\xi_3,\xi_4) = \whnu{_{\alpha}^{(n)}}(\xi_3,\xi_4).
\]
By Fourier inversion,
\[
\frac{1}{p^{2k}}\sum_{(\xi_3,\xi_4)\in\FF_{p^k}^2} \whnu{_{\alpha}^{(n)}}(\xi_3,\xi_4) = \Pr(R(\alpha)=R'(\alpha)=0) \le \Pr(R(\alpha)=0).
\]
Combining the above equalities proves the proposition.
\end{proof}

We therefore need to estimate the probability that $R$ has a root which is inadmissible. We will show that, after excluding some primes, these probabilities are bounded by $n^{-A}$ for $A>0$. We shall use the following finite-group version of the inverse Littlewood--Offord theorem. Unfortunately, the result in \cite[Theorem~1.6]{KoenigNguyenPan24} is stated for lazy Bernoulli variables. However, a standard argument allow to apply it to any fixed non-degenerate distribution on $\Z$. Let us give the defailts.
\begin{defn}\label{def:CCCP} Let $G$ be an Abelian group. A symmetric coset progression is a set of the form $H+Q$ where $H$ is a subgroup of $G$ and $Q=\phi(\Lambda\cap \Z^r)$, where $\Lambda$ is a symmetric box in $\R^r$ and $\phi$ is a group homomorphism from $\Z^r$ to $G$. It is called proper if $\phi$, restricted to $\Lambda\cap \Z^r$, is one-to-one. The number $r$ is called the rank of the symmetric coset progression.

  A centered convex progression is similar, except $\Lambda$ does not need to be a box, but can be any symmetric convex set in $\R^r$.
\end{defn}
\begin{lem}\label{lem:finite-group-ILO}
For every $A>0$ and $0<\eps<1$ there are constants $r,B,m_0=O_{A,\eps,\mu}(1)$ such that the following holds for every $m\ge m_0$. Let $\omega_1,\dots,\omega_m$ be i.i.d.\ $\mu$-random variables. Then for any prime $p$ outside a finite set depending only on~$\mu$, for any $k\ge 1$, and for any $\alpha_1,\dots,\alpha_m\in\FF_{p^k}$, if
\[
\sup_{\beta\in\FF_{p^k}}\bigg|\Pr\bigg(\sum_{j=1}^m\omega_j \alpha_j = \beta\bigg) - \frac{1}{p^k}\bigg| \ge m^{-A},
\]
then there is a proper symmetric coset progression $H+Q\subseteq\FF_{p^k}$ of rank at most $r$ and cardinality at most $m^B$ which contains all but at most $\eps m$ of the $\alpha_j$, counted with multiplicity.
\end{lem}

\begin{proof}
  We may assume without loss of regularity that $\supp(\mu)$ is not contained in a coset of a proper subgroup of $\Z$ (indeed, the set of excluded primes is exactly the divisors of the maximal $d$ such that $\supp(\mu)\subset a+d\Z$). This gives that $|\widehat{\mu}(\xi)|$ takes the value 1 only at $\xi=0$ and hence $|\widehat{\mu}(\xi)|\le 1-c\sin^2(\xi/2)$ for some $c$ (we normalise the Fourier transform to have $\xi\in[0,1]$).

Fix $0<\lambda_{\mu}\le\frac{1}{4}$ such that $4\pi^2\lambda_{\mu}\le c_{\mu}$, and let $\eta_1,\dots,\eta_m$ be independent lazy Bernoulli random variables with parameter $\lambda_{\mu}$, so that $\Pr(\eta_j=0)=1-\lambda_{\mu}$ and $\Pr(\eta_j=1)=\Pr(\eta_j=-1)=\frac{\lambda_{\mu}}{2}$. Then %
\[
  \widehat{\eta_1}(\xi) = 1-\lambda_{\mu}+\lambda_{\mu}\cos(\xi) = 1-2\lambda_{\mu}\sin^2(\xi/2)
  \ge |\widehat{\mu}(\xi)|.
\] %
Writing $\tr$ for the trace from $\FF_{p^k}$ to $\FF_p$, Fourier inversion gives
\begin{align*}
  \lefteqn{\sup_{\beta\in\FF_{p^k}} \bigg|\Pr\bigg(\sum_{j=1}^m\omega_j\alpha_j=\beta\bigg) - \frac{1}{p^k}\bigg| \le \frac{1}{p^k} \sum_{\xi\in\FF_{p^k}^{\times}} \prod_{j=1}^m\left|\widehat{\mu}\left(\frac{\tr(\xi\alpha_j)}{p}\right)\right|}\qquad & \\
  &\le \frac{1}{p^k} \sum_{\xi\in\FF_{p^k}^{\times}} \prod_{j=1}^m\widehat{\eta_1}\left(\frac{\tr(\xi\alpha_j)}{p}\right)
  = \Pr\bigg(\sum_{j=1}^m\eta_j\alpha_j=0\bigg) - \frac{1}{p^k} \\
  &\le \sup_{\beta\in\FF_{p^k}}\bigg|\Pr\bigg(\sum_{j=1}^m\eta_j\alpha_j=\beta\bigg) - \frac{1}{p^k}\bigg|.
\end{align*}
This shows that
\[
\sup_{\beta\in\FF_{p^k}}\bigg|\Pr\bigg(\sum_{j=1}^m\eta_j\alpha_j=\beta\bigg) - \frac{1}{p^k}\bigg| \ge m^{-A}.
\]

For all sufficiently large $m$, the fixed parameter $\lambda_{\mu}$ satisfies the assumptions of \cite[Theorem~1.6(i)]{KoenigNguyenPan24}, with $\eps_0=\frac{1}{4}$. Taking the number of exceptional elements to be $\floor{\eps m}$, that theorem gives a proper symmetric coset progression of bounded rank containing all but at most $\eps m$ of the~$\alpha_j$. Its cardinality is bounded by $m^B$ for some $B=O_{A,\eps,\mu}(1)$.
\end{proof}

\begin{prop}\label{prop:inadmissible-root}
For any $A,\lambda>0$ there are constants $C,d_0,m_0=O_{A,\lambda,\mu}(1)$ such that the following holds. For every $m\ge m_0$, there are at most $m^C$ primes $p$ for which there exist $k\ge 1$ and $\alpha\in\FF_{p^k}^{\times}$ of multiplicative order $d$ with $d_0\le d\le \lambda m$, such that
\[
\sup_{\beta\in\FF_{p^k}}\left|\Pr(R_{1,m}(\alpha)=\beta)-\frac{1}{p^k}\right| \ge \frac{1}{m^A}.
\]
\end{prop}

We emphasize that the exceptional set is deterministic, i.e., depends only on $m,A,\lambda$ and $\mu$, but not on $R$.

\begin{proof}
We begin by applying the inverse Littlewood--Offord \Lref{lem:finite-group-ILO} with $\eps=\frac{1}{2}$ to the elements $\alpha,\dots,\alpha^m$. After discarding a finite set of primes depending only on $\mu$, there is a proper symmetric coset progression $H+Q\sub\FF_{p^k}$ of rank at most $r=O_{A,\mu}(1)$ and cardinality at most $m^B$, where $B=O_{A,\mu}(1)$, such that $\alpha^j\in H+Q$ for at least $\frac{1}{2}m$ values of $1\le j\le m$.

We discard all primes which are at most $m^B$. For any remaining prime we have $|H|\le |H+Q|\le m^B<p$. Since every nontrivial additive subgroup of $\FF_{p^k}$ has cardinality at least $p$, it follows that $H=\set{0}$. Therefore $\alpha^j\in Q$ for at least $\frac{1}{2}m$ values of $1\le j\le m$. Each power of $\alpha$ can appear at most $\ceil{m/d}$ times among $\alpha,\dots,\alpha^m$. In particular, $Q$ contains at least $\frac{m}{2\ceil{m/d}} \ge \frac{1}{4}\min\set{d,m}$ distinct powers of $\alpha$.

Let $M=\prod_{\ell\le r+1}\ell=O_{A,\mu}(1)$ be the product of all primes $\ell\le r+1$. We claim that we can find $r+1$ distinct powers $\alpha^{t_1},\dots,\alpha^{t_{r+1}}\in Q$ such that $\alpha^{M(t_j-t_{j'})}\ne 1$ for every $j\ne j'$. Indeed, define an equivalence relation on the distinct powers of~$\alpha$ in $Q$ by $\alpha^{t_j}\sim\alpha^{t_{j'}}$ if $\alpha^{M(t_{j'}-t_j)}=1$. Every equivalence class has cardinality at most $\gcd(M,d)\le M$. Hence there are at least $\frac{1}{4M}\min\set{d,m}$ equivalence classes. We may therefore choose $d_0\ge 4M(r+1)$ so that, for all sufficiently large $m$, there are at least $r+1$ such classes. Choosing one power from each of $r+1$ distinct classes proves the claim. We take $0\le t_j<d$.

Write $Q=\set{m_1g_1+\cdots+m_sg_s\suchthat |m_i|\le N_i}$, where $s\le r$. Since $Q$ is proper, $\prod_{i=1}^s(2N_i+1)=|Q|\le m^B$, and hence $N_i\le m^B$ for every $i$. Write $\alpha^{t_j}=m_{j,1}g_1+\cdots+m_{j,s}g_s$. The corresponding $r+1$ vectors in $\Z^s$ are linearly dependent, so there are primitive integers $c_1,\dots,c_{r+1}$, not all zero, such that $\sum_{j=1}^{r+1}c_j\alpha^{t_j}=0$ in $\FF_{p^k}$. By taking a minimal linear dependence and expressing its coefficients as minors of the coordinate matrix, we have $|c_j|\le s!m^{Bs}$, so we may moreover assume that $|c_j|\le m^{B'}$ for some $B'=O_{A,\mu}(1)$. After discarding the terms for which $c_j=0$, define $F(X)=\sum_jc_jX^{t_j}$. Since $\alpha$ has multiplicative order $d$ and $p\nmid d$, we have $\Phi_d(\alpha)=0$. Thus $F$ and $\Phi_d$ have a common root over $\FF_{p^k}$, and hence $p\mid\Res(\Phi_d,F)$.

We claim that $\Res(\Phi_d,F)\ne 0$. Otherwise, $F(\zeta)=0$ for some primitive $d$-th root of unity $\zeta\in\C$. Take a minimal vanishing subsum of $F(\zeta)=\sum_jc_j\zeta^{t_j}=0$. This subsum has at most $r+1$ terms. By a result of Mann \cite{Mann65} (see also the improvement of Conway--Jones \cite{ConwayJones76}), the order of the quotient of any two roots occurring in this subsum divides $M$. In other words, for some $j\ne j'$, $d\, |\, M(t_j-t_{j'})$. Since %
$\alpha$ has order $d$, this implies $\alpha^{M(t_j-t_{j'})}=1$, contrary to the choice of $t_1,\dots,t_{r+1}$. Therefore $\Res(\Phi_d,F)\ne 0$.

It remains to count the possible prime divisors of these resultants. Since $\Phi_d$ is monic,
\[
\left|\Res(\Phi_d,F)\right| = \prod_{\substack{\zeta\in\C\\ \ord(\zeta)=d}}|F(\zeta)| \le \bigg(\sum_j|c_j|\bigg)^{\varphi(d)} \le \left((r+1)m^{B'}\right)^{\varphi(d)}.
\]
Since $d\le \lambda m$, it follows that $\left|\Res(\Phi_d,F)\right|\le m^{Cm}$ for some $C=O_{A,\lambda,\mu}(1)$. Thus each nonzero resultant has at most $Cm\log m$ distinct prime divisors.

Finally, there are only polynomially many possibilities for $d$, for the exponents $t_1,\dots,t_{r+1}$, and for the coefficients $c_1,\dots,c_{r+1}$: their number is at most $m^{O_{A,\mu}(1)}$. Summing the number of prime divisors over all these possibilities, and adding the primes $p\le m^B$ and the finite set of primes excluded at the beginning, gives at most~$m^C$ possible primes, after increasing $C$ and $m_0$ if necessary.
\end{proof}

We now fix the constant $B_0$ of Assumption~\ref{assum:P-prop}.

\begin{rem}\label{rem:B0}
  From now on, let $d_0=O_{\mu}(1)$ be the constant of \Pref{prop:inadmissible-root} with $A=3$ and $\lambda=1$. %
  Fix $B_0\ge\max\set{3,2d_0+1}$. This is the constant we use for Assumption~\ref{assum:P-prop}.
\end{rem}

The above proposition does not handle elements with multiplicative order smaller than $d_0$. To cover them as well, we use the following deterministic lemma:

\begin{lem}\label{lem:too-small-order}
Let $P\in\Z[x]$ be a polynomial of degree $n$ and height at most $H$. Fix $d_0\ge 1$, and let $p$ be a prime with $p>(Hn(n+1))^{d_0}$. If $P$ has a double root in $\overline{\FF_p}$ of multiplicative order $d<d_0$, then $\Phi_d^2\mid P$ in $\Z[x]$. In particular, $\disc(P)=0$.
\end{lem}

\begin{proof}
By the assumption, $p\mid\Res(\Phi_d,P)$. On the other hand,
\[
\left|\Res(\Phi_d,P)\right| = \prod_{\substack{\zeta\in\C \\ \ord(\zeta)=d}}\left|P(\zeta)\right| \le (H(n+1))^{\varphi(d)} < p
\]
by the assumption. Therefore $\Res(\Phi_d,P)=0$, so $\Phi_d\mid P$ in $\Z[x]$. Write $P=\Phi_dQ$ for some $Q\in\Z[x]$. Similarly, $p\mid\Res(\Phi_d,P')$ and
\[
\left|\Res(\Phi_d,P')\right| = \prod_{\substack{\zeta\in\C \\ \ord(\zeta)=d}}\left|P'(\zeta)\right| \le (Hn(n+1))^{\varphi(d)} < p,
\]
so $\Phi_d\mid P'=\Phi_d'Q+\Phi_dQ'$. Since $\gcd(\Phi_d,\Phi_d')=1$, we must have $\Phi_d\mid Q$, so we conclude $\Phi_d^2\mid P$.
\end{proof}

For a prime $p$ and a positive integer $k\ge 1$, we write $B_{p^k}^{\nad}$ for the event that~$R$ has a double root $\alpha\in\BB_{p^k}$ (recall that this means that $\ord(\alpha)<n$). We can now conclude:

\begin{cor}\label{cor:no-inadmissible-double-roots}
There is a constant $C>0$, depending only on $\mu$, such that, for all sufficiently large $n$, there is a set $\PP_0=\PP_0(n,\mu)$ of at most $Cn^C$ primes with the following property. Let~ $\PP$ be a finite set of primes disjoint from $\PP_0$ satisfying \ref{item:P1}. Also, let $\ksm\ge 2$ (sm standing for small) be a positive integer such that
\[
p^{\ksm} \le \exp\left(\frac{\sqrt{n}}{\log^3 n}\right)
\]
for every $p\in\PP$. Then
\[
\Pr\bigg(\set{\disc(R)\ne 0}\cap \bigcup_{p\in\PP}\bigg(B_p^{\ad}\cap \bigcup_{k=1}^{\ksm} B_{p^k}^{\nad}\bigg)\bigg) \le \frac{3(1+o(1))}{n}S(\PP).
\]
\end{cor}
(Recall the definition of $B_p^{\ad}$, Definition \ref{def:Bad}).
\begin{proof}
Let $C,m_0,d_0$ be the constants of \Pref{prop:inadmissible-root} for $A=3$ and $\lambda=1$, and let $\PP_0$ be the set of exceptional primes (from the same proposition). Then $\PP_0$ has size at most $Cn^C$. Since we take the intersection with $\set{\disc(R)\ne 0}$, we may discard all elements of multiplicative order smaller than $d_0$ by \Lref{lem:too-small-order} (here we use $B_0\ge 2d_0+1$, so all primes in $\PP$, for sufficiently large $n$, are larger than $(Hn(n+1))^{d_0}$).

Let $\PP$ be a set of primes such that $\PP\cap\PP_0=\varnothing$. Then, for any $p\in\PP$, we have $\frac{1}{p^k}\le\frac{1}{p}\le\frac{1}{n^3}$, so \Pref{prop:inadmissible-root} shows that for any $k\ge 1$ and any $\alpha\in\FF_{p^k}$ of multiplicative order $d$ satisfying $d_0\le d < n$,
\[
\sup_{\beta\in\FF_{p^k}}\Pr(R_{1,n-1}(\alpha) = \beta) \le \frac{1}{p^k} + \frac{1}{(n-1)^3} \le \frac{3}{n^3}.
\]
It follows that
\[
\Pr(R(\alpha)=0) = \sum_{a_0\in\Z}\Pr(R_{1,n-1}(\alpha)=-a_0-\alpha^n)\Pr(\omega_0=a_0) \le \frac{3}{n^3},
\]
so by \Pref{prop:double-Fp-Fq}, we have
\begin{multline*}
  \Pr(R(a)=R'(a)=R(\alpha)=R'(\alpha)=0) \\
  \le \exp\left(-\frac{cn}{\log p^{2k+2}(\log\log p^{2k+2})^5}\right) + \frac{3}{p^2n^3}\le \frac{3+o(1)}{p^2n^3}
\end{multline*}
where the last inequality follows from our assumption that $p^k\le\exp(\sqrt{n}/\log^3n)$ and from the fact that $n$ is sufficiently large.

Finally, using \Lref{lem:count-inadmissible}, we have
\begin{align*}
  & \Pr\bigg(\set{\disc(R)\ne 0}\cap \bigcup_{p\in\PP}\bigg(B_p^{\ad}\cap \bigcup_{k=1}^{\ksm} B_{p^k}^{\nad}\bigg)\bigg) \\
  &\qquad\le \sum_{p\in\PP}\sum_{k=1}^{\ksm}\sum_{a\textrm{ admissible}}\sum_{\substack{\alpha\in\BB_{p^k} \\ \ord(\alpha)\ge d_0}} \Pr(R(a)=R'(a)=R(\alpha)=R'(\alpha)=0) \\
  &\qquad\le \sum_{p\in\PP}\frac{3+o(1)}{pn} = \frac{3+o(1)}{n}S(\PP).
\end{align*}
This completes the proof.
\end{proof}

\section{Double roots in large field extensions}\label{sec:BSG}

The last ingredient that we need is to bound the probability that $R$ has a double root in larger field extensions than the previous methods provide. We prove this using tools from additive number theory. %

Recall that $\nu_{a,\alpha}$ denotes the law of the random quadrupule $(R(a),R'(a),R(\alpha),R'(\alpha))$, where $a \in \FF_p$ and $\alpha \in \FF_{p^k}$.

We begin with the following bound on the weight $\nu_{a,\alpha}$ assigns to affine subspaces:

\begin{lem}\label{lem:affine-subspace}
Let $p$ be a prime, let $2\le k\le \frac{m-2}{2}$ be a positive integer, let $a\in\FF_p^{\times}$, and let~$\alpha\in\FFF_{p^k}$. Assume that $\supp(\mu)\sub(-p/2,p/2)$. Then, for every affine $\FF_p$-subspace $U\sub\FF_p^2\times\FF_{p^k}^2$ of dimension $d\le 2k+2$,
\[
\nu_{a,\alpha}^{(1,m)}(U) \le \norm{\mu}_{\infty}^{2k+2-d}.
\]
\end{lem}

\begin{proof}
Consider the linear map $\Phi\colon\FF_p^{2k+2}\to\FF_p^2\times\FF_{p^k}^2$ given by
\[
\Phi(c_1,\dots,c_{2k+2}) = \left(\sum_{j=1}^{2k+2} c_j a^j, \sum_{j=1}^{2k+2} jc_j a^{j-1}, \sum_{j=1}^{2k+2} c_j\alpha^j, \sum_{j=1}^{2k+2} jc_j\alpha^{j-1}\right).
\]
We claim that $\Phi$ is injective. Indeed, if $\Phi(c_1,\dots,c_{2k+2})=0$, then $P(a)=P'(a)=P(\alpha)=P'(\alpha)=0$ for $P(x)=\sum_{j=1}^{2k+2} c_jx^j$. Hence $(x-a)^2f_{\alpha}^2\mid P$, where $f_{\alpha}$ is the minimal polynomial of $\alpha$ over $\FF_p$. But $\deg f_{\alpha}=k$ and $\deg P\le 2k+2$, so this forces $P=0$ or $P=c(x-a)^2f_{\alpha}^2$. The latter is impossible since $P(0)=0$, so we must have $P=0$, proving that~$\Phi$ is injective.

Condition on the values of $\omega_{2k+3},\dots,\omega_m$, and put
\[
y=\left(\sum_{j=2k+3}^m\omega_j a^j,\sum_{j=2k+3}^m j\omega_j a^{j-1},\sum_{j=2k+3}^m\omega_j \alpha^j,\sum_{j=2k+3}^m j\omega_j\alpha^{j-1}\right).
\]
We need to bound the probability that $\Phi(\omega_1,\dots,\omega_{2k+2})\in U-y$. Let $S\coloneqq\Phi^{-1}(U-y)$. If $S$ is empty there is nothing to prove. Otherwise, $S$ is an affine subspace of $\FF_p^{2k+2}$ of some dimension $d_0\le d$.

Choose a set $J\sub\set{1,\dots,2k+2}$ of size~$d_0$ such that the coordinate projection $\pi_J\colon\FF_p^{2k+2}\to\FF_p^J$ is one-to-one on~$S$. For every $u\in\pi_J(S)$, let $c(u)$ denote the unique element of $S$ such that $\pi_J(c(u))=u$. It follows that
\begin{align*}
  \Pr\bigl((\omega_1,\dots,\omega_{2k+2})\in S\bigr) &= \sum_{u\in\pi_J(S)} \prod_{j=1}^{2k+2}\mu(c_j(u)) \le \norm{\mu}_\infty^{2k+2-d_0}\sum_{u\in\pi_J(S)}\prod_{j\in J}\mu(u_j) \\
  &\le \norm{\mu}_\infty^{2k+2-d_0} \le \norm{\mu}_\infty^{2k+2-d}.
\end{align*}
The bound is uniform in the conditioned values of $\omega_{2k+3},\dots,\omega_m$. Averaging over these coefficients proves the lemma.
\end{proof}

The main tool that we will use is the following application of Balog--Szemer\'{e}di--Gowers, and of the quasi-polynomial version of Freiman's theorem by Sanders.

\begin{prop}\label{prop:dichotomy}
Let $m\ge 1$ and $2\le k\le \frac{m-2}{2}$ be positive integers, let $a\in\FF_p^{\times}$, and let $\alpha\in\FFF_{p^k}$. Suppose that $\supp(\mu)\sub(-p/2,p/2)$.

Then, for any $K>2$,
\[
\text{either} \qquad \norm{\nu_{a,\alpha}^{(1,2m)}}_2 < \frac{1}{K}\norm{\nu_{a,\alpha}^{(1,m)}}_2 \qquad \text{or} \qquad \norm{\nu_{a,\alpha}^{(1,m)}}_2^2 \le \frac{1}{p^{2k + 2 - D_K}}
\]
for a nonnegative real number $D_K=O_{\mu}(\log^{3+o(1)}K)$.
\end{prop}

\begin{proof}
For convenience, write $\nu=\nu_{a,\alpha}^{(1,m)}$ and $A=A_{(a,\alpha)}$ of \Lref{lem:linear-transform}, so that $\nu_{a,\alpha}^{(m+1,2m)} = (A^m)_*\nu$. Then $\norm{\nu_{a,\alpha}^{(m+1,2m)}}_2 = \norm{\nu}_2$, and $\nu_{a,\alpha}^{(1,2m)} = \nu*((A^m)_*\nu)$. Suppose that
\[
\norm{\nu_{a,\alpha}^{(1,2m)}}_2 \ge \frac{1}{K}\norm{\nu}_2,
\]
i.e.,
\[
\norm{\nu*((A^m)_*\nu)}_2 \ge \frac{1}{K}\norm{\nu}_2.
\]
We apply Varj\'{u}'s version of the Balog--Szemer\'{e}di--Gowers theorem (\cite[Lemma 15]{Varju12}) with $\nu_1=\nu$ and $\nu_2=\nu_{a,\alpha}^{(m+1,2m)}$. We get a subset $H\sub\FF_p^2\times\FF_{p^k}^2$ such that
\[
\frac{1}{CK^C\norm{\nu}_2^2} \le \left|H\right| \le \frac{CK^C}{\norm{\nu}_2^2},
\]
\begin{equation}\label{eq:HHH}
\left|H+H+H\right| \le K^C\left|H\right|,
\end{equation}
and
\[
\min_{y\in H}(\check{\nu}*\nu)(y) \ge \frac{1}{CK^C\left|H\right|}
\]
where $\check{\nu}(y)=\nu(-y)$. Since
\[
\sum_{y\in H}(\check{\nu}*\nu)(y) = \sum_{w\in\FF_p^2\times\FF_{p^k}^2}\nu(w)\nu(H+w),
\]
it follows that
\[
\nu(H+w) \ge \frac{1}{CK^C}
\]
for some $w\in\FF_p^2\times\FF_{p^k}^2$.

We next apply Sanders' quantitative version of Freiman's theorem \cite[Theorem~1.4]{Sanders13}. See also \cite{Raghavan25} for stronger results (the difference does not matter for our application). We apply it to the set $H$, and get from \eqref{eq:HHH} that there exist a set $X\sub\FF_p^2\times\FF_{p^k}^2$, a centered convex progression $P\sub\FF_p^2\times\FF_{p^k}^2$ (recall Definition \ref{def:CCCP}), and an $\FF_p$-subspace $V\le\FF_p^2\times\FF_{p^k}^2$, with
\begin{equation}\label{eq:dichotomy-dimP}
\dim P \le \log^{3+o(1)}K,
\end{equation}
\[
\left|P+V\right| \le \exp(\log^{3+o(1)}K)\left|H\right|,
\]
and
\[
\left|X\right| \le \exp(\log^{3+o(1)}K),
\]
such that
\[
H \sub X+P+V.
\]
We claim that
\begin{equation}\label{eq:dichotomy-dimV}
\dim V \ge 2k+2 - D_K\qquad\text{where}\qquad D_K=O_{\mu}(\log^{3+o(1)}K).
\end{equation}
Indeed, if $\dim P + \dim V > 2k+2$, this follows from \eqref{eq:dichotomy-dimP} by taking $D_K=\dim P$. Otherwise, assume $\dim P + \dim V \le 2k+2$. Then $P+V$ (and thus also $P+V+x+w$ for any $x\in X$) is contained in an affine subspace of dimension at most $\dim P + \dim V \le 2k+2$, so by \Lref{lem:affine-subspace}
\[
\nu(P+V+x+w) \le \norm{\mu}_{\infty}^{2k+2 - \dim P - \dim V}.
\]
for any $x\in X$. On the other hand, by the lower bound on $\nu(H+w)$, there exists an element $x_0\in X$ such that
\[
\nu(P+V+x_0+w) \ge \frac{1}{CK^C\left|X\right|}.
\]
Comparing the two bounds for $\nu(P+V+x_0+w)$ and taking logarithms, we obtain
\[
\dim V \ge 2k+2 - \dim P - \frac{\log(CK^C|X|)}{\log\left(\norm{\mu}_{\infty}^{-1}\right)} \ge 2k+2 - D_K
\]
for $D_K \coloneqq \dim P + \frac{\log(CK^C|X|)}{\log\left(\norm{\mu}_{\infty}^{-1}\right)}\le C_{\mu}\log^{3+o(1)}K$, where $C_{\mu}>0$ depends only on $\norm{\mu}_{\infty}$, proving \eqref{eq:dichotomy-dimV}.

Finally,
\[
p^{2k+2-D_K} \le p^{\dim V} = \left|V\right| \le \left|P+V\right| \le \exp\left(\log^{3+o(1)}K\right)\left|H\right| \le \frac{K_0}{\norm{\nu}_2^2}
\]
for
\[
K_0 \coloneqq CK^C\exp\left(\log^{3+o(1)}K\right) = \exp(\log^{3+o(1)}K).
\]
Therefore,
\[
\norm{\nu}_2^2 \le \frac{K_0}{p^{2k+2 - D_K}} \le \frac{1}{p^{2k+2 - D_K - \log K_0 / \log p}},
\]
and the proposition follows by replacing $D_K$ with $D_K + \frac{\log K_0}{\log p}$.
\end{proof}

In the next proposition, we bound the probability that $R$ has an admissible double root in $\FF_p$ and a double root in $\FFF_{p^k}$. We write $B_{p^k}$ for the event that $R$ has a double root in $\FFF_{p^k}$ (which may be admissible or inadmissible).

\begin{prop}\label{prop:no-double-roots-big-ext}
Let $\PP$ be a finite set of primes satisfying \ref{item:P1}. Set $m\coloneqq\floor{\frac{n-1}{3}}$ and, for each $p\in\PP$,
\begin{equation}\label{eq:def kmu}
k_{\mu}(p) \coloneqq \max\set{k\ge 1 \suchthat p^{k+2} \le \norm{\mu}_2^{-2m}}.
\end{equation}
Then, if $n$ is large enough,
\[
\Pr\left(\bigcup_{p\in\PP}\bigcup_{k=\ceil{\log^4 n}}^{k_{\mu}(p)}(B_p^{\ad}\cap B_{p^k})\right) \le \frac{2}{n^2}S(\PP) + 2S_2(\PP) + o\left(\frac{1}{\log n}\right).
\]
\end{prop}

\begin{proof}
For convenience, we write $k_0\coloneqq\ceil{\log^4 n}$.

Fix a prime $p\in\PP$ and an integer $k_0\le k\le k_{\mu}(p)$. We note that
\[
2k+2 \le \frac{2\log\norm{\mu}_2^{-2m}}{\log p} = \frac{2H_2(\mu)}{\log p}m \le m
\]
by our assumptions. For any $a\in\AA_p$ and $\alpha\in\FFF_{p^k}$,
\[
\nu_{a,\alpha}^{(n)} = \nu_{a,\alpha}^{(1,2m)}*\nu_{a,\alpha}^{(2m+1,3m)}*\lambda
\]
(for $\lambda$ which corresponds to all coefficients outside of $[1,3m]$), and thus by Young's inequality
\[
\Pr(R(a)=R'(a)=R(\alpha)=R'(\alpha)=0) \le \norm{\nu_{a,\alpha}^{(n)}}_{\infty} \le \norm{\nu_{a,\alpha}^{(1,2m)}}_2 \norm{\nu_{a,\alpha}^{(2m+1,3m)}}_2.
\]
Recalling that
\[
\norm{\nu_{a,\alpha}^{(2m+1,3m)}}_2 = \norm{\nu_{a,\alpha}^{(1,m)}}_2
\]
we have
\[
\Pr(R(a)=R'(a)=R(\alpha)=R'(\alpha)=0) \le \norm{\nu_{a,\alpha}^{(1,m)}}_2 \norm{\nu_{a,\alpha}^{(1,2m)}}_2,
\]
and thus
\begin{equation}\label{eq:double-root-pk}
  \Pr(B_p^{\ad}\cap B_{p^k}) \le \sum_{a\in\AA_p}\sum_{\alpha\in\FFF_{p^k}}\norm{\nu_{a,\alpha}^{(1,m)}}_2 \norm{\nu_{a,\alpha}^{(1,2m)}}_2.
\end{equation}

We set $K=n^3$, and assume that $n$ is sufficiently large so that, for $D_K$ of \Pref{prop:dichotomy}, we have
\[
D_K \le C_{\mu}\log^{3+o(1)}K \le \log^4 n - 1.
\]
By \Pref{prop:dichotomy}, for any $a\in\AA_p$ and $\alpha\in\FFF_{p^k}$,
\[
\text{either} \qquad \norm{\nu_{a,\alpha}^{(1,2m)}}_2 < \frac{1}{n^3}\norm{\nu_{a,\alpha}^{(1,m)}}_2 \qquad \text{or} \qquad \norm{\nu_{a,\alpha}^{(1,m)}}_2^2 \le \frac{1}{p^{2k+3 - \log^4 n}}.
\]

We also recall that
\[
\norm{\nu_{a,\alpha}^{(1,2m)}}_2 \le \norm{\nu_{a,\alpha}^{(1,m)}}_2.
\]
It therefore follows that
\[
\norm{\nu_{a,\alpha}^{(1,m)}}_2\norm{\nu_{a,\alpha}^{(1,2m)}}_2 \le \frac{1}{n^3}\norm{\nu_{a,\alpha}^{(1,m)}}_2^2 + \frac{1}{p^{2k + 3 - \log^4 n}}
\]
for any $a\in\AA_p$ and any $\alpha\in\FFF_{p^k}$. Substituting back in \eqref{eq:double-root-pk}, we have
\begin{align*}
  \Pr(B_p^{\ad}\cap B_{p^k}) &\le \sum_{a\in\AA_p}\sum_{\alpha\in\FFF_{p^k}} \left(\frac{1}{n^3}\norm{\nu_{a,\alpha}^{(1,m)}}_2^2 + \frac{1}{p^{2k+3 - \log^4 n}}\right) \\
  &\le \frac{1}{p^{k+2 - \log^4 n}} + \frac{1}{n^3}\sum_{a\in\AA_p}\sum_{\alpha\in\FFF_{p^k}} \norm{\nu_{a,\alpha}^{(1,m)}}_2^2
\end{align*}

Let $\widetilde{R}$ be an independent copy of $R$, and write $S=R_{1,m}-\widetilde{R}_{1,m}$. This is a random polynomial with coefficient law $\mu_S=\mu*\check{\mu}$, where $\check{\mu}(t)=\mu(-t)$. Therefore $\norm{\mu_S}_2\le\norm{\mu}_2$, and $\supp(\mu_S)\sub[-2H_{\mu},2H_{\mu}]$. For any prime $p\in\PP$ and any $k\ge 2$, we write $\DD_{p,p^k}(S)$ for the number of pairs $(a,\alpha)\in\AA_p\times\FFF_{p^k}$ such that $a,\alpha$ are double roots of $S$. Then
\begin{align*}
  \sum_{a\in\AA_p}\sum_{\alpha\in\FFF_{p^k}} \norm{\nu_{a,\alpha}^{(1,m)}}_2^2 &= \sum_{a\in\AA_p}\sum_{\alpha\in\FFF_{p^k}} \Pr\left(\begin{array}{c}R_{1,m}(a)=\widetilde{R}_{1,m}(a), R_{1,m}'(a)=\widetilde{R}_{1,m}'(a), \\ R_{1,m}(\alpha)=\widetilde{R}_{1,m}(\alpha), R_{1,m}'(\alpha)=\widetilde{R}_{1,m}'(\alpha)\end{array}\right) \\
  &= \sum_{a\in\AA_p}\sum_{\alpha\in\FFF_{p^k}} \Pr\left(\begin{array}{c} S(a)=S'(a)=0 \\ S(\alpha)=S'(\alpha)=0\end{array}\right) = \E[\DD_{p,p^k}(S)],
\end{align*}
and thus
\[
\Pr(B_p^{\ad}\cap B_{p^k}) \le \frac{1}{n^3}\E[\DD_{p,p^k}(S)] + \frac{1}{p^{k+2 - \log^4 n}}.
\]

Summing over $k$, we have
\[
\sum_{k=k_0}^{k_{\mu}(p)}\Pr(B_p^{\ad}\cap B_{p^k}) \le \frac{1}{n^3}\sum_{k=k_0}^{k_{\mu}(p)}\E[\DD_{p,p^k}(S)] + \sum_{k=k_0}^{k_{\mu}(p)}\frac{1}{p^{k+2 - \log^4 n}}.
\]
The second term is simple to estimate:
\[
\sum_{k=k_0}^{k_{\mu}(p)}\frac{1}{p^{k+2 - \log^4 n}} \le \frac{2}{p^{k_0+2 - \log^4 n}} \le \frac{2}{p^2}.
\]
To estimate the remaining sum, we replace the sum over double roots in $\FF_{p^k}$ with a crude bound of $n$, leaving only the randoness in the roots $a\in\FF_p$. Formally, if $S\ne 0$, then
\begin{equation}\label{eq:Dpp}
\sum_{k=k_0}^{k_{\mu}(p)}\DD_{p,p^k}(S) \le n\sum_{a\in\AA_p}\ind{\set{S(a)=S'(a)=0}},
\end{equation}
On the other hand, if $S=0$, which happens with probability $\Pr(S=0)=\norm{\mu}_2^{2m}$, then $\DD_{p,p^k}(S)\le p^{k+1}$. Therefore, by \Cref{cor:mix-app1} (applied to $S$),
\begin{align*}
  \sum_{k=k_0}^{k_{\mu}(p)}\E[\DD_{p,p^k}(S)] &= \E\left[\sum_{k=k_0}^{k_{\mu}(p)}\DD_{p,p^k}(S)\right] \\
  &\le \E\left[\sum_{k=k_0}^{k_{\mu}(p)}\DD_{p,p^k}(S)\ind{\set{S\ne 0}}\right] + \sum_{k=k_0}^{k_{\mu}(p)}p^{k+1}\Pr(S=0) \\
  &\stackrel{\mathclap{(*)}}{\le} n\E\left[\sum_{a\in\AA_p}\ind{\set{S(a)=S'(a)=0}}\right] + 2p^{k_{\mu}(p)+1}\norm{\mu}_2^{2m} \\
  &\stackrel{\mathclap{(**)}}{\le} n\sum_{a\in\AA_p}\Pr(S(a)=S'(a)=0) + \frac{2}{p} \\
  &\stackrel{\mathclap{(\dagger)}}{\le} n\sum_{a\in\AA_p}\left(\frac{1}{p^2} + \exp\left(-c\frac{(1-\norm{\mu}_2^2)\sqrt{n}}{\log^2 n}\right)\right) + \frac{2}{p} \\
  &\le \frac{n}{p} + np\exp\left(-c\frac{(1-\norm{\mu}_2^2)\sqrt{n}}{\log^2 n}\right) + \frac{2}{p},
\end{align*}
where in $(*)$ we used \eqref{eq:Dpp} for the left term, in $(**)$ we used \eqref{eq:def kmu} for the right term, and in $(\dagger)$ we used Corollary \ref{cor:mix-app1} (applied to $S$).
Consequently, for sufficiently large $n$
\begin{align*}
  \sum_{k=k_0}^{k_{\mu}(p)}\Pr(B_p^{\ad}\cap B_{p^k}) &\le \frac{1}{n^2p} + \frac{p}{n^2}\exp\left(-c\frac{(1-\norm{\mu}_2^2)\sqrt{n}}{\log^2 n}\right) + \frac{2}{n^3p} + \frac{2}{p^2} \\
  &\le \frac{2}{n^2p} + \frac{2}{p^2} + \exp\left(-c\frac{(1-\norm{\mu}_2^2)\sqrt{n}}{2\log^2 n}\right).
\end{align*}
Finally, summing over $p$,
\begin{align*}
  \sum_{p\in\PP}\sum_{k=k_0}^{k_{\mu}(p)}\Pr(B_p^{\ad}\cap B_{p^k}) &\le \frac{2}{n^2}S(\PP) + 2S_2(\PP) + \left|\PP\right|\exp\left(-c\frac{(1-\norm{\mu}_2^2)\sqrt{n}}{2\log^2 n}\right),
\end{align*}
and the last term is $o\left(\frac{1}{\log n}\right)$ since $\left|\PP\right|\le\exp\left(\frac{C\sqrt{n}}{\log^3 n}\right)$.
\end{proof}

\section{Proof of the main theorem}\label{sec:proof}

For the proof of \Tref{thm:main}, we need the following observation:

\begin{prop}\label{prop:p^2-divides-disc}
Let $P\in\Z[x]$ be a monic polynomial with $\disc(P)\ne 0$, and let $p$ be an odd prime. We write $\overline{P}$ for the reduction of $P$ modulo $p$. Suppose that $a\in\FF_p$ is a double root of $\overline{P}$. Then $p^2\mid\disc(P)$ if and only if at least one of the following holds:
\begin{enumerate}
  \item\label{item:p^2-i} $a$ is a triple root of $\overline{P}$;
  \item\label{item:p^2-ii} there is a monic irreducible polynomial $Q\ne x-a$ such that $Q^2\mid\overline{P}$, in which case $v_p(\disc(P))\ge 1+\deg Q$;
  \item\label{item:p^2-iii} $a$ lifts to a double root of $P$ modulo $p^2$.
\end{enumerate}
\end{prop}

Here $v_p(\disc(P))$ denotes the $p$-adic valuation of $\disc(P)$.

Recall that $\Z_p$ are the $p$-adic integers.

\begin{lem}Let $P\in\Z[x]$ be a monic polynomial with $\disc(P)\ne 0$, and let $p$ be a prime. Consider the ideal $(P,P')$ generated by $P$ and $P'$ in $\Z_p[x]$. Then \[v_p(\disc(P)) = v_p(|\Z_p[x]/(P,P')|).\]
\end{lem}

\begin{proof}
Consider the self-map $T_P\colon A \mapsto P'A$ on the $\Z_p$-module $M:=\Z_p[x]/(P)$.  Let $\Q_p$ be the field of $p$-adic fractions. We note that $\disc(P)\ne 0$ also over $\Q_p$, as the embedding $\Q\to\Q_p$ is one-to-one. If $\alpha_1,\ldots,\alpha_n$ are the roots of $P$ in an algebraic closure $\overline{\Q_p}$, they are distinct since $\disc(P)\neq 0$, so the map $A \mapsto (A(\alpha_1),\ldots,A(\alpha_n))$ realizes an isomorphism between $M\otimes \overline{\Q_p}$ and $\overline{\Q_p}^n$ (here $\otimes$ is as $\Z_p$-modules). Under this isomorphism $T_P$ becomes diagonal with eigenvalues $P'(\alpha_1),\ldots,P'(\alpha_n)$. In particular, $\det(T_P)=\prod_i P'(\alpha_i)=(-1)^{n(n-1)/2}\disc(P)$. So $v_p(\disc(P))=v_p(\det(T_P))$. On the other hand, $\Z_p[x]/(P,P')\cong M/\mathrm{Im}(T_P)$.
By Smith normal form over $\Z_p$, there are bases of the domain and codomain in which $T_P$ is diagonal with entries
$p^{a_1},\ldots,p^{a_n}$, where $a_i\ge 0$. Thus
\[
|M/\mathrm{Im}(T_P)| = \prod_{i=1}^n|\Z_p/p^{a_i}\Z_p| = p^{\sum_{i=1}^n a_i} = p^{v_p(\det(T_P))},
\]
hence the result follows.
\end{proof}

\begin{proof}[Proof of \Pref{prop:p^2-divides-disc}]
From the previous lemma, it is clear that
\[
v_p(\disc(P)) \ge \deg\gcd(\overline{P},\overline{P}').
\]
Consequently, if there is another repeated irreducible factor $Q\ne x-a$ of $\overline{P}$, we get $v_p(\disc(P))\ge 1+\deg Q$. This proves the second part of (\ref{item:p^2-ii}). Also, if (\ref{item:p^2-i}) or (\ref{item:p^2-ii}) holds, then $v_p(\disc(P))\ge 2.$

Suppose then that neither (\ref{item:p^2-i}) nor (\ref{item:p^2-ii}) holds. We need to show that (\ref{item:p^2-iii}) holds if and only if $v_p(\disc(P))\ge 2.$ Choose a lift $a_0\in\Z_p$ of $a$ in the $p$-adic integers $\Z_p$. By Hensel's lemma, we may write $P=P_1P_2$ over $\Z_p$, where $P_1(x+a_0)=x^2+Bx+C$ (for $B,C\in p\Z_p$), such that $P_2(a_0)$, $\disc(P_2)$, and $\Res(P_1,P_2)$, are units in $\Z_p$. Hence
\[
v_p(\disc(P))=v_p(\disc(P_1))=v_p(B^2-4C).
\]
Since $p$ is odd and $B^2\in p^2\Z_p$, the assumption $p^2\mid\disc(P)$ is equivalent to $C\in p^2\Z_p$, i.e. $P(a_0)\equiv 0\pmod{p^2}$. By \Lref{lem:double-root-lift} this happens if and only if (\ref{item:p^2-iii}) holds.
\end{proof}

We are ready to prove the main theorem.

\begin{proof}[Proof of \Tref{thm:main}]
The results of \cite{FeldheimSen17} show that $\Pr(\disc(R)=0)\ll 1/n$ under the assumption that $\mu(i)<\frac12$ for all $i\in\Z$ (see also the recent \cite{MichelenYakir26} for non-quantitative results without any restriction on $\mu$). This condition is much weaker than ours, as $\mu(i)\ge 1/2$ implies $H_2(\mu)\le \log 4$. Thus it is enough to show that $\Pr(\disc(R)\textrm{ is a nonzero square})\le C/\log n$.

Let $\PP_0$ denote the set of exceptional primes of \Cref{cor:no-inadmissible-double-roots}, which has size at most $Cn^C$ for some $C=C(\mu)>0$. Let $\PP_1$ denote the set of all primes which lie between~$n^{B_0}$ and $\exp\left(\frac{\sqrt{n}}{\log^7 n}\right)$, and set $\PP=\PP_1\setminus\PP_0$. Then $\PP$ satisfies \ref{item:P1}. We further set $\ksm=\floor{\log^4 n}$, so that
\[
p^{\ksm} \le \exp\left(\frac{\sqrt{n}}{\log^3 n}\right)
\]
for every $p\in\PP$.

By Mertens' theorem,
\[
S(\PP_1) = \sum_{p\in\PP_1}\frac{1}{p} = \log\log\exp\left(\frac{\sqrt{n}}{\log^7 n}\right) - \log\log(n^{B_0}) + o(1) = \left(\frac{1}{2}+o(1)\right)\log n.
\]
On the other hand,
\begin{align*}
  S(\PP_0\cap\PP_1) &= \sum_{p\in\PP_0\cap\PP_1}\frac{1}{p} \le \sum_{n^{B_0}\le p\le n^{B_0+C+1}}\frac{1}{p} + \sum_{\substack{p\in\PP_0 \\ p>n^{B_0+C+1}}}\frac{1}{p} \\
  &\le O_{\mu}(1) + \frac{Cn^C}{n^{B_0+C+1}} = O_{\mu}(1).
\end{align*}
Therefore $S(\PP) = \left(\frac{1}{2}+o(1)\right)\log n$.

We choose $\eps>0$ such that $\eps<\frac{1}{2}-\frac{6}{H_2(\mu)}$ and write $m=\floor{\frac{n-1}{3}}$. Combining previous results in the paper, we see that with probability at least $1-\frac{C'}{\log n}$ (where $C'$ depends only on $\mu,\eps$), either $\disc(R)=0$ or the following hold:
\begin{enumerate}
  \item\label{item:proof1} There are at least $\left(\frac{1}{2}-\eps\right)\log n$ primes in $\PP$ such that there is a unique admissible double root of $R$ in $\FF_p$ (\Pref{prop:num-unique-double} with $\delta=\frac{\eps}{3}$);
  \item\label{item:proof2} There are no admissible triple roots of $R$ in $\FF_p$ for any $p\in\PP$ (\Pref{prop:no-triple-roots});
  \item \label{item:proof3} No admissible double root of $R$ lifts to a double root modulo $p^2$ for any $p\in\PP$ (\Pref{prop:no-double-root-lifts});
  \item\label{item:proof4} There are no pairs of admissible double roots $(a,\alpha)\in\AA_p\times\AA_{p^k}$ of $R$ for any $p\in\PP$ and $2\le k\le\log^4 n$ (\Pref{prop:no-admissible-double-roots-ext});
  \item\label{item:proof5} There are no pairs of double roots $(a,\alpha)\in\AA_p\times\BB_{p^k}$ of $R$ for any $p\in\PP$ and $1\le k\le\log^4 n$, even for indamissible $\alpha$ (\Cref{cor:no-inadmissible-double-roots});
  \item\label{item:proof6} There are no pairs of double roots $(a,\alpha)\in\AA_p\times\FFF_{p^k}$ of $R$ for any $p\in\PP$ and any $k>\log^4 n$ with $p^{k+2}\le\norm{\mu}_2^{-2m}$, whether $\alpha$ is admissible or not (\Pref{prop:no-double-roots-big-ext}).
\end{enumerate}

Suppose that $\disc(R)$ is a nonzero square, and let $\PP(R)$ denote a set of precisely $\floor{\left(\frac{1}{2}-\eps\right)\log n}$ primes satisfying (\ref{item:proof1}). Fix a prime $p\in\PP(R)$. By (\ref{item:proof1}), $R$ has a double root $a\in\FF_p^{\times}$, so $p$ must divide $\disc(R)$. By our assumption, $\disc(R)$ is a square, so $p^2\mid\disc(R)$. Since $B_0\ge 3$ we have $p>n$, we may use \Pref{prop:p^2-divides-disc} to deduce that one of the conclusions of \Pref{prop:p^2-divides-disc} holds. The first conclusion is excluded by (\ref{item:proof2}), and the third by (\ref{item:proof3}). Therefore there is some irreducible $P\in\FF_p[x]$ such that $P\ne x-a$ and $P^2\mid R$ in $\FF_p[x]$.

We note that $P\ne x$. Indeed, we may assume that  $n$ is sufficiently large so that $p>H$ for all $p\in\PP$. Then $x^2\mid R$ in $\FF_p[x]$ forces $\omega_0=\omega_1=0$ as integers. But then $x^2\mid R$ in $\Z[x]$, which is impossible since we assumed $\disc(R)\ne 0$.

Suppose next that $\deg P=1$, i.e., $P=x-b$ for some $b\in\FF_p$ with $b\ne 0,a$. If $b$ is admissible, then $a,b$ are two admissible double roots of $R$, contradicting (\ref{item:proof1}); if $b$ is inadmissible, we get a contradiction from (\ref{item:proof5}). This shows that $\deg P\ge 2$.

Let $k=\deg P$, and let $\alpha$ be a root of $P$. Then $\alpha$ is a double root of $R$ in $\FFF_{p^k}$. By (\ref{item:proof4}) and (\ref{item:proof5}), we must have $k > \log^4 n$. But then by (\ref{item:proof6}) we have
\[
p^{k+2} > \norm{\mu}_2^{-2m} = \exp(H_2(\mu)m),
\]
so
\[
v_p(\disc(R))\ge 1+k \ge \floor{\frac{H_2(\mu)m}{\log p}}.
\]

We can now conclude the proof. We define the (random) integer
\[
N = \prod_{p\in\PP(R)} p^{\floor{\frac{H_2(\mu)m}{\log p}}} \ge \prod_{p\in\PP(R)} p^{\frac{H_2(\mu)m}{\log p} - 1} = \frac{\exp\left(H_2(\mu)m\left|\PP(R)\right|\right)}{\prod_{p\in\PP(R)}p}.
\]
By the above, $N\mid\disc(R)$, and
\begin{align*}
  \log N &\ge H_2(\mu)m\left|\PP(R)\right| - \sum_{p\in\PP(R)}\log p \\
  &\ge H_2(\mu)\floor{\frac{n-1}{3}}\floor{\left(\frac{1}{2}-\eps\right)\log n} - \left(\frac{1}{2}-\eps\right)\frac{\sqrt{n}\log n}{\log^7 n} \\
  &= \left(\frac{1}{6}-\frac{\eps}{3}+o(1)\right)H_2(\mu)n\log n.
\end{align*}
On the other hand, since $H(R)\le H$, we have
\[
\left|\disc(R)\right| \le (Hn)^{2n} = \exp(2n\log n + 2n\log H),
\]
which is a contradiction when $n$ is sufficiently large since $\left(\frac{1}{6}-\frac{\eps}{3}\right)H_2(\mu)>2$.
\end{proof}

\bibliographystyle{plain}
\bibliography{refs}

\begin{thebibliography}{10}

\bibitem{BarySorokerGoldgraber25}
Lior Bary-Soroker and Noam Goldgraber.
\newblock Full {G}alois groups of polynomials with slowly growing coefficients.
\newblock {\em Bull. Lond. Math. Soc.}, 57(3):941--955, 2025.

\bibitem{BarySorokerKoukoulopoulosKozma23}
Lior Bary-Soroker, Dimitris Koukoulopoulos, and Gady Kozma.
\newblock Irreducibility of random polynomials: general measures.
\newblock {\em Invent. Math.}, 233(3):1041--1120, 2023.

\bibitem{BarySorokerKozma20}
Lior Bary-Soroker and Gady Kozma.
\newblock Irreducible polynomials of bounded height.
\newblock {\em Duke Math. J.}, 169(4):579--598, 2020.

\bibitem{Bhargava25}
Manjul Bhargava.
\newblock Galois groups of random integer polynomials and van der {W}aerden's
  conjecture.
\newblock {\em Ann. of Math. (2)}, 201(2):339--377, 2025.

\bibitem{BlochPolya31}
A.~Bloch and G.~P\'olya.
\newblock On the {R}oots of {C}ertain {A}lgebraic {E}quations.
\newblock {\em Proc. London Math. Soc. (2)}, 33(2):102--114, 1931.

\bibitem{BreuillardVarju19}
Emmanuel Breuillard and P\'eter~P. Varj\'u.
\newblock Irreducibility of random polynomials of large degree.
\newblock {\em Acta Math.}, 223(2):195--249, 2019.

\bibitem{BreuillardVarju22}
Emmanuel Breuillard and P\'eter~P. Varj\'u.
\newblock Cut-off phenomenon for the {$ax+b$} {M}arkov chain over a finite
  field.
\newblock {\em Probab. Theory Related Fields}, 184(1-2):85--113, 2022.

\bibitem{Chela63}
R.~Chela.
\newblock Reducible polynomials.
\newblock {\em J. London Math. Soc.}, 38:183--188, 1963.

\bibitem{ConwayJones76}
J.~H. Conway and A.~J. Jones.
\newblock Trigonometric {D}iophantine equations ({O}n vanishing sums of roots
  of unity).
\newblock {\em Acta Arith.}, 30(3):229--240, 1976.

\bibitem{Dietmann13}
Rainer Dietmann.
\newblock Probabilistic {G}alois theory.
\newblock {\em Bull. Lond. Math. Soc.}, 45(3):453--462, 2013.

\bibitem{Dobrowolski79}
E.~Dobrowolski.
\newblock On a question of {L}ehmer and the number of irreducible factors of a
  polynomial.
\newblock {\em Acta Arith.}, 34(4):391--401, 1979.

\bibitem{FeldheimSen17}
Ohad~N. Feldheim and Arnab Sen.
\newblock Double roots of random polynomials with integer coefficients.
\newblock {\em Electron. J. Probab.}, 22:Paper No. 10, 23, 2017.

\bibitem{Gallagher73}
P.~X. Gallagher.
\newblock The large sieve and probabilistic {G}alois theory.
\newblock In {\em Analytic number theory ({P}roc. {S}ympos. {P}ure {M}ath.,
  {V}ol. {XXIV}, {S}t. {L}ouis {U}niv., {S}t. {L}ouis, {M}o., 1972)}, volume
  Vol. XXIV of {\em Proc. Sympos. Pure Math.}, pages 91--101. Amer. Math. Soc.,
  Providence, RI, 1973.

\bibitem{Hilbert1892}
David Hilbert.
\newblock Ueber die {I}rreducibilit\"at ganzer rationaler {F}unctionen mit
  ganzzahligen {C}oefficienten.
\newblock {\em J. Reine Angew. Math.}, 110:104--129, 1892.

\bibitem{Hokken26}
David Hokken.
\newblock Counting (skew-) reciprocal littlewood polynomials with square
  discriminant.
\newblock {\em Israel Journal of Mathematics}, pages 1--33, 2026.

\bibitem{HokkenKoukoulopoulos25}
David Hokken and Dimitris Koukoulopoulos.
\newblock {Irreducibility and Galois groups of random reciprocal polynomials of
  large degree}.
\newblock {\em arXiv preprint arXiv:2510.18857}, 2025.

\bibitem{Kac43}
M.~Kac.
\newblock On the average number of real roots of a random algebraic equation.
\newblock {\em Bull. Amer. Math. Soc.}, 49:314--320, 1943.

\bibitem{KoenigNguyenPan24}
Jake Koenig, Hoi~H. Nguyen, and Amanda Pan.
\newblock A note on inverse results of random walks in abelian groups.
\newblock {\em Comb. Number Theory}, 13(1):67--92, 2024.

\bibitem{Konyagin92}
S.~V. Konyagin.
\newblock Estimates for {G}aussian sums and {W}aring's problem modulo a prime.
\newblock {\em Trudy Mat. Inst. Steklov.}, 198:111--124, 1992.

\bibitem{Konyagin99}
S.~V. Konyagin.
\newblock On the number of irreducible polynomials with {$0,1$} coefficients.
\newblock {\em Acta Arith.}, 88(4):333--350, 1999.

\bibitem{Mahler64}
K.~Mahler.
\newblock An inequality for the discriminant of a polynomial.
\newblock {\em Michigan Math. J.}, 11:257--262, 1964.

\bibitem{Mann65}
Henry~B. Mann.
\newblock On linear relations between roots of unity.
\newblock {\em Mathematika}, 12:107--117, 1965.

\bibitem{MichelenYakir25}
Marcus Michelen and Oren Yakir.
\newblock Law of large numbers for the discriminant of random polynomials.
\newblock {\em arXiv preprint arXiv:2506.12206}, 2025.

\bibitem{MichelenYakir26}
Marcus Michelen and Oren Yakir.
\newblock Limit law for root separation in random polynomials.
\newblock {\em Adv. Math.}, 496:Paper No. 110990, 98, 2026.

\bibitem{OdlyzkoPoonen93}
A.~M. Odlyzko and B.~Poonen.
\newblock Zeros of polynomials with {$0,1$} coefficients.
\newblock {\em Enseign. Math. (2)}, 39(3-4):317--348, 1993.

\bibitem{Raghavan25}
Rushil Raghavan.
\newblock Improved bounds for the freiman-ruzsa theorem.
\newblock {\em arXiv preprint arXiv:2512.11217}, 2025.

\bibitem{Sanders13}
Tom Sanders.
\newblock The structure theory of set addition revisited.
\newblock {\em Bull. Amer. Math. Soc. (N.S.)}, 50(1):93--127, 2013.

\bibitem{vanderWaerden36}
B.~L. van~der Waerden.
\newblock Die {S}eltenheit der reduziblen {G}leichungen und der {G}leichungen
  mit {A}ffekt.
\newblock {\em Monatsh. Math. Phys.}, 43(1):133--147, 1936.

\bibitem{Varju12}
P\'eter~P. Varj\'u.
\newblock Expansion in {$SL_d(\mathscr O_K/I)$}, {$I$} square-free.
\newblock {\em J. Eur. Math. Soc. (JEMS)}, 14(1):273--305, 2012.

\end{thebibliography}

\end{document}